\documentclass[11pt]{amsart}
\usepackage[margin=1in]{geometry}

\usepackage{amsmath, amsfonts, amssymb, amsthm}
\usepackage{amsrefs}
\usepackage{mathrsfs}
\usepackage{enumitem} 
\usepackage{hyperref}
\usepackage{color}
\usepackage{dsfont}
\usepackage{csquotes}
\usepackage{epigraph}
\usepackage{comment}
\usepackage{mlmodern}
\usepackage[T1]{fontenc}

\usepackage{color}
\definecolor{darkblue}{rgb}{0.0,0.0,0.3}
\hypersetup{
    colorlinks=false,
    linkcolor=blue,
    urlcolor=darkblue,
    }

\newtheorem{theorem}{Theorem}[section]

\newtheorem{lemma}[theorem]{Lemma}
\newtheorem{corollary}[theorem]{Corollary}

\theoremstyle{definition}

\newtheorem{definition}[theorem]{Definition}

\newtheorem{notation}[theorem]{Notation}

\theoremstyle{remark}
\newtheorem{remark}[theorem]{Remark}

\numberwithin{equation}{section}

\newcommand{\bR}{\mathbb{R}}

\newcommand{\cD}{\mathcal{D}}

\newcommand{\Ampere}{Amp\`{e}re}

\newcommand{\Garding}{G\r{a}rding}
\newcommand\norm[1]{\left\lVert#1\right\rVert}
\DeclareMathOperator*{\osc}{osc}

\title[Interior curvature estimates]{Interior curvature estimates for convex solutions of a prescribed curvature quotient equation}
\author{Bin Wang}
\dedicatory{}
\address[]{Department of Mathematics, The Chinese University of Hong Kong, Hong Kong, China} 
\curraddr{Institute for Theoretical Sciences, Westlake University, Hangzhou, China}
\email{wangbin@westlake.edu.cn}

\begin{document}

\setcounter{tocdepth}{1} 
\begin{abstract}
In this note, we prove a priori interior curvature bounds for convex solutions to an elliptic Weingarten curvature quotient equation. The proof avoids integral estimates, doubling inequalities and compactness arguments. It relies on a concavity inequality for the equation operator and a novel choice of the auxiliary function to carry out an elementary maximum-principle argument. As an application, we obtain interior curvature bounds for convex solutions of a prescribed special Lagrangian curvature equation in dimensions three and four.
\end{abstract}
\maketitle
\tableofcontents

\section{Introduction}
Let $f$ be a smooth symmetric function of $n$ variables and let $\Omega \subseteq \bR^n$ be a bounded domain. For the Hessian matrix $D^2u$ of a given smooth function $u$, we let $\lambda(D^2u)=(\lambda_1,\ldots,\lambda_n)$ denote its eigenvalues. For a smooth hypersurface $\Sigma$, we let $\kappa[\Sigma]=(\kappa_1,\ldots,\kappa_n)$ denote its principal curvatures.
Under some natural structure conditions on the function $f$, Caffarelli, Nirenberg, and Spruck \cite{CNS3,CNS4,CNS5} conducted a systematic treatment for the Hessian equations,
\begin{equation}
F(D^2u):=f(\lambda(D^2u))=\psi(x), \quad x \in \Omega, \label{Hessian equations}
\end{equation} and the Weingarten curvature equations,
\begin{equation}
H(D^2u, Du):=f(\kappa[\Sigma])=\psi(x), \quad x \in \Omega, \label{Weingarten equations}
\end{equation} where $\psi$ is a given smooth positive function on $\overline{\Omega}$. Similar investigations were also carried out by Ivochkina \cite{Ivochkina-1, Ivochkina-2} and Krylov \cite{Krylov-1987}. These fully nonlinear equations have been extensively studied since then. Among all choices for the symmetric function $f$, the most typical ones are the $k$-th elementary symmetric polynomials
\begin{align*}
\sigma_k(\mu_1,\ldots,\mu_n)&:=\sum_{1 \leq i_1<\cdots<i_k \leq n} \mu_{i_1}\cdots \mu_{i_k}, \quad 1 \leq k \leq n, \quad \mu \in \bR^n,\\
\sigma_0(\mu_1,\ldots,\mu_n)&:=1,
\end{align*} and their quotients
\[\frac{\sigma_k}{\sigma_l}, \quad 1 \leq l<k \leq n.\] For these choices, both the Hessian equations and the Weingarten equations are elliptic at $k$-admissible solutions.

\begin{definition}
Let $0 \leq l < k \leq n$ and let $f=\sigma_k/\sigma_l$.
We say a function $u \in C^2(\Omega) \cap C(\overline{\Omega})$ is $k$-admissible in $\Omega$ if the eigenvalues $\lambda(D^2u)=(\lambda_1,\ldots,\lambda_n)$ of its Hessian $D^2u$ satisfy
\[\lambda(D^2u) \in \Gamma_k:=\{\lambda \in \bR^n: \sigma_{j}(\lambda)>0 \quad \forall\ 1 \leq j \leq k\}\] for all $x \in \Omega$. When $\Sigma$ is the graph of $u$ over $\Omega$, for a solution $u$ of the Weingarten curvature equation \eqref{Weingarten equations}, we say $u$ is $k$-admissible if the principal curvatures $\kappa[\Sigma]=(\kappa_1,\ldots,\kappa_n)$ satisfy $\kappa[\Sigma] \in \Gamma_k$ at every point of its graph. 
\end{definition}

In this note, we shall be concerned with the class of prescribed curvature quotient equations,
\[\left(\frac{\sigma_k}{\sigma_l}\right)(\kappa[u])=\psi(x,u(x)), \quad x \in \Omega,\] where $\kappa[u]=(\kappa_1,\ldots,\kappa_n)$ denotes the principal curvatures of the graph of $u$.
One major reason for investigating such equations is due to their natural appearance in geometric problems. In particular, for principal curvatures $\kappa=(\kappa_1,\ldots,\kappa_n)$ of a smooth hypersurface,
\[
\begin{alignedat}{2}
\sigma_1(\kappa)
  &= \sum_{i=1}^{n}\kappa_i
  &\quad& \text{is the mean curvature,}\\
\sigma_2(\kappa)
  &= \sum_{i<j}\kappa_i\kappa_j
  && \text{is the scalar curvature,}\\
\sigma_n(\kappa)
  &= \prod_{i=1}^{n}\kappa_i
  && \text{is the Gauss curvature,}\\
\left(\frac{\sigma_n}{\sigma_{n-1}}\right)(\kappa)
  &= \left(\sum_{i=1}^{n}\frac{1}{\kappa_i}\right)^{-1}
  && \text{is the harmonic curvature.}
\end{alignedat}
\] Moreover, for other values of $k$ and $l$, they can still represent remarkable geometric information, e.g., in convex geometry and conformal geometry. From our limited knowledge, we point out the references \cite{Guan-Guan, Guan-Li-Li, Guan-Li, Guan-Ma} and \cite{BV,GVW,V} for some of their appearance in the respective subject. 

Before we state the main theorem, let us recall some results on the interior regularity for admissible solutions of the Hessian quotient equations,
\begin{equation}
\left(\frac{\sigma_k}{\sigma_l}\right)(\lambda(D^2u))=\psi(x,u(x)), \quad x \in \Omega.\label{Hessian quotient equations}
\end{equation} In \cite{Lu-2}, Lu derived interior Hessian estimates for convex solutions of the equation \eqref{Hessian quotient equations} when $(k,l)=(n,n-1)$ and $(k,l)=(n,n-2)$. This was the first major progress on the subject; see also \cite{Fung} and \cite{Jianxiang-2} for alternative proofs. In addition, Lu showed the existence of non-classical Lipschitz continuous solutions for the quotient equation \eqref{Hessian quotient equations} when $k-l \geq 3$. It is also known that the Hessian estimates may fail for $k-l \in \{1,2\}$ with $k \leq n$ if the right-hand side $\psi$ contains the gradient $Du$; see \cite[Remark 1.4]{Fung} and \cite[Theorem 1.2]{Guan-Ren-Wang}. Therefore, it remains to study the quotient \eqref{Hessian quotient equations} for $k-l \in \{1,2\}$ with $k<n$ and $\psi=\psi(x,u)$. 

In a subsequent work \cite{LT-CPAA}, Lu and Tsai demonstrated that if a certain concavity inequality could be derived for the quotient operators $\sigma_k/\sigma_{k-1}$ and $\sigma_k/\sigma_{k-2}$, then Lu's method in \cite{Lu-2} can be carried out to obtain interior Hessian estimates for convex solutions of those two quotients. The required concavity inequality for $\sigma_k/\sigma_{k-1}$ was later derived by Li-Wu \cite{Li-Wu} and Tsai \cite{Tsai} for convex solutions, independently using different methods. Li and Wu \cite{Li-Wu} also derived the concavity inequality for convex solutions to the quotient $\sigma_k/\sigma_{k-2}$. In \cite{Dong-Zhang}, Dong and Zhang were able to derive both inequalities for semi-convex solutions, i.e., solutions with $D^2u \geq -KI$ for some $K \geq 0$. Hence, by the framework of Lu-Tsai \cite{LT-CPAA}, interior Hessian estimates for semi-convex solutions of the two ``middle quotients'' follow as a consequence. Whether one can derive the interior estimates for $k$-admissible solutions was unclear at the time of writing this manuscript. Some other partial progress were obtained by Jiao-Sui \cite{Jiao-Sui}, Lu-Sroka \cite{Lu-Sroka}, Mei-Yan \cite{Mei-Yan}, Wang-Yuan \cite{Wang-Yuan}, and Zhou \cite{Zhou}. In \cite{LT-Crelle}, Lu and Tsai derived Pogorelov type interior Hessian estimates for convex solutions of the quotient $\sigma_n/\sigma_l$, $1 \leq l \leq n-1$. A more thorough discussion on related literature is deferred to Section \ref{Literature} below.

In this note, we consider the following prescribed curvature quotient equation
\[\left(\frac{\sigma_k}{\sigma_{k-2}}\right)(\kappa[u])=\psi(x,u(x)), \] and derive interior curvature estimates for convex solutions. 

\begin{theorem} \label{our theorem 2}
Let $3 \leq k \leq n-1$ and let $r >0$. For a given positive function $\psi(x,u) \in C^{1,1}(\overline{B_r} \times \bR)$, suppose $u \in C^4(B_r)$ is a $k$-admissible and convex solution of the equation
\[H(D^2u,Du):=\left( \frac{\sigma_k}{\sigma_{k-2}} \right)(\kappa[u])=\psi(x,u(x)) \quad \text{in $B_r$},\] where $\kappa[u]=(\kappa_1,\ldots,\kappa_n)$ are the principal curvatures of its graph with respect to the outward unit normal,
then
\begin{equation}
\sup_{x \in B_{r/2}} \kappa_{\max}(x) \leq C \label{the estimate}
\end{equation} for some constant $C>0$ whose value depends only on $n, k, r, \norm{u}_{L^\infty(B_r)}, \norm{Du}_{L^{\infty}(B_r)}$, $(\inf_{\cD} \psi)^{-1}$ and $\norm{\psi}_{C^{1,1}(\cD)}$, where
\[\cD:=\overline{B_{\frac{3r}{4}}} \times [-M, M], \quad M:=1+\norm{u}_{L^{\infty}(B_r)}.\]
\end{theorem}

Several remarks are in order.

\begin{remark}
By extending the method of Guan-Qiu \cite{Guan-Qiu} for the $\sigma_2$ curvature equation, Liu \cite{Jianxiang} derived the interior estimates \eqref{the estimate} for $k=n$. One of our contributions is the further extension of their methods to get the estimates for $k<n$. When $3 \leq k \leq n-1$, the quotient operator $\sigma_k/\sigma_{k-2}$ no longer possesses some convenient algebraic properties as $\sigma_2$ and $\sigma_n/\sigma_{n-2}$. To make Guan-Qiu's delicate analysis applicable for the quotient $\sigma_k/\sigma_{k-2}$, one subtlety is to obtain \eqref{claim}; see Remark \ref{concluding remark 2}.
\end{remark}

\begin{remark}
As we have mentioned above, Lu and Tsai \cite{LT-CPAA} showed how to use an integral approach to obtain interior Hessian estimates for convex solutions of the Hessian quotient $(\sigma_k/\sigma_{k-2})(\lambda(D^2u))$, if a concavity inequality was available. The required inequality have now been verified by Li-Wu \cite{Li-Wu} and Dong-Zhang \cite{Dong-Zhang}, so the interior estimates would follow as a consequence of Lu-Tsai's framework in \cite{LT-CPAA}. Although our Theorem \ref{our theorem 2} is a ``curvature counterpart'' of their results, some aspects of our result differ from theirs in character. First, whether the integral approach of Lu-Tsai \cite{LT-CPAA} would work for the curvature quotient equations is unknown because the structure of the curvature equations are substantially different; see \eqref{the graph curvature matrix}. It might be expected to work fine but no one has done the attempt, so we cannot say this for sure and hence the proof of the estimate \eqref{the estimate} does not automatically follow from Lu-Tsai \cite{LT-CPAA}. Second, our proof for Theorem \ref{our theorem 2} employs the classical Bernstein technique in the sense that we apply the maximum principle to a suitable auxiliary function and proceed with a series of standard computations. Hence, our proof is quite elementary in contrast to the existing advanced methods involving integral estimates, doubling inequalities, and compactness arguments. The proof particularly relies on the concavity inequality for the quotient operator $\sigma_k/\sigma_{k-2}$ that were derived by Li-Wu \cite{Li-Wu} and Dong-Zhang \cite{Dong-Zhang}, and the novel auxiliary function of Guan-Qiu \cite{Guan-Qiu}. 
\end{remark}

\begin{remark}
$u$ being convex gives us that $\kappa_{i} \geq 0$ for all $1 \leq i \leq n$, which should not be confused with the condition $\kappa[u] \in \Gamma_n$.
In Theorem \ref{our theorem 2}, we may replace the condition ``$k$-admissible and convex'' by ``($k+1$)-admissible'' alone. In other words, our proof in fact works for $(k+1)$-admissible solutions, i.e., solutions $u$ with $\kappa[u] \in \Gamma_{k+1}$; see Remark \ref{concluding remark 1}.
Note the inclusion relation of the \Garding\ cones,
\[\{\kappa \in \bR^n: \kappa_i>0 \quad \forall\ 1 \leq i \leq n\}=\Gamma_n \subseteq \cdots \subseteq \Gamma_{k+1} \subseteq \Gamma_{k} \subseteq \cdots \subseteq \Gamma_{1}.\]  Our result is thus close to the ultimate goal of deriving the interior estimates \eqref{the estimate} for $k$-admissible solutions.
\end{remark}

\begin{remark}
Interior curvature estimates for $k$-admissible solutions of the other quotient equation
\[\left(\frac{\sigma_k}{\sigma_{k-1}}\right)(\kappa[u])=\psi(x,u(x))\] follow from Sheng-Urbas-Wang's remarkable proof of Pogorelov type interior curvature estimates for a large class of curvature functions $f(\kappa[u])$; see \cite[Remark 2.1]{Sheng-Urbas-Wang}.
\end{remark}

Finally, since the quotient operator $\sigma_k/\sigma_{k-2}$ ``coincides'' with the special Lagrangian operator in low dimensions, we deduce the following.

\begin{corollary} \label{our corollary}
Let $n=3,4$ and let $r>0$. For a given positive function $\psi(x,u) \in C^{1,1}(\overline{B_r} \times \bR)$, suppose $u \in C^4(B_r)$ is a convex solution of the equation
\begin{equation}
\sum_{i=1}^{n} \arctan \left[\frac{\kappa_i(x)}{\psi(x,u(x))}\right]=\pi \quad \text{in $B_r$}, \label{the SL curvature equation}
\end{equation} where $\kappa[u]=(\kappa_1,\ldots,\kappa_n)$ are the principal curvatures of its graph with respect to the outward unit normal, then
\[
\sup_{x \in B_{r/2}} \kappa_{\max}(x) \leq C 
\] for some constant $C>0$ whose value depends on $n, r, \norm{u}_{L^\infty(B_r)}, \norm{Du}_{L^{\infty}(B_r)}$, $(\inf_{\cD} \psi)^{-1}$ and $\norm{\psi}_{C^{1,1}(\cD)}$, where
\[\cD:=\overline{B_{\frac{3r}{4}}} \times [-M, M], \quad M:=1+\norm{u}_{L^{\infty}(B_r)}.\]
\end{corollary}

\begin{remark}
When $\psi \equiv 1$, this result was proved in all dimensions by Qiu and Zhou \cite{Qiu-Zhou} using an advanced integral method. Hence, our proof also serves as an elementary alternative for their result in dimensions $n=3,4$.
\end{remark}

The notion of the special Lagrangian curvature was defined by Smith in \cite{Smith}, which we now recall.
Let $n \geq 2$. For a smooth hypersurface $\Sigma$ with non-negative principal curvatures $\kappa_1,\ldots,\kappa_n$, if $q \geq 1$ of the curvatures are positive, then given $\Theta \in (0,q\pi/2)$, for each $X \in \Sigma$, following the terminology in \cite[page 59]{Smith}, we define the $\Theta$-special Lagrangian curvature $R_{\Theta}(X)>0$ to be the unique number such that
\[\sum_{i=1}^{n} \arctan(R_\Theta \cdot \kappa_i)=\Theta.\] Therefore, the equation \eqref{the SL curvature equation} prescribes the $\pi$-special Lagrangian curvature of the graph of $u$ by the function $1/\psi$, that is,
\[R_{\pi}(X)=\psi^{-1}(X).\]

The rest of this note is organized as follows. In Section \ref{preliminaries} we collect standard properties of the $\sigma_k$ operator and introduce some notations. In Section \ref{our theorem 2 proof} and Section \ref{corollary proof}, we prove Theorem \ref{our theorem 2} and Corollary \ref{our corollary}, respectively. 
The reader may go straight to the case-by-case analysis in Section \ref{part 2 of the proof} which is the essence of our proof, if they are familiar with the proof of Guan-Qiu \cite{Guan-Qiu}. Finally, Section \ref{Literature} contains a thorough discussion on related literature and existing methods.

\textit{A note.} In an earlier version of this manuscript, we had assumed the validity of the concavity inequality \eqref{concavity} and we showed how to use it for deriving the interior curvature estimates \eqref{the estimate}. After the manuscript was submitted for peer review, we found the preprints of Li-Wu \cite{Li-Wu} and Dong-Zhang \cite{Dong-Zhang} in which they proved the required inequality. We thus no longer need to make that assumption and have modified the contents into the current version.

\section{Preliminaries} \label{preliminaries}
Recall the definition of the $k$-th elementary symmetric polynomial,
\[\sigma_k(\kappa_1,\ldots,\kappa_n):=\sum_{1 \leq i_1<\cdots<i_k \leq n} \kappa_{i_1}\cdots \kappa_{i_k}, \quad \kappa=(\kappa_1,\ldots,\kappa_n) \in \bR^n,\] and we adopt the convention that $\sigma_0:=1$ and $\sigma_k:=0$ for $k>n$. The associated $k$-th \Garding\ cone is an open symmetric convex cone defined by 
\[\Gamma_k:=\{\kappa \in \bR^n: \sigma_j(\kappa)>0 \quad \forall\ 1 \leq j \leq k\}.\]

\begin{notation}
Observe that
\[\frac{\partial}{\partial \kappa_i}\sigma_k(\kappa)=\sigma_{k-1}(\kappa)\bigg|_{\kappa_i=0}=\sigma_{k-1}(\kappa_1,\ldots,\kappa_{i-1},0,\kappa_{i+1},\ldots,\kappa_n).\] Therefore, we may use $\sigma_{k-1}(\kappa|i)$ to denote the first order derivatives. The notation $\sigma_{k-2}(\kappa|ij)$ is defined in a similar fashion for the second order derivatives. 
\end{notation}

The following are some commonly used properties of the $\sigma_k$ operator and we may state them without proofs.
\begin{lemma} \label{sigma_k properties 1}
For all $1 \leq k \leq n$ and $\kappa \in \bR^n$, we have
\begin{align*}
\sigma_k(\kappa)&=\kappa_i\sigma_{k-1}(\kappa|i)+\sigma_{k}(\kappa|i),\\
\sum_{i=1}^{n}\kappa_i\sigma_{k-1}(\kappa|i)&=k\sigma_k(\kappa), \\
\sum_{i=1}^{n} \sigma_{k-1}(\kappa|i)&=(n-k+1)\sigma_{k-1}(\kappa).
\end{align*}
\end{lemma}

\begin{lemma} \label{sigma_k properties 2}
Let $2 \leq k \leq n$.
Assume $\kappa \in \Gamma_k$ is ordered as
\[\kappa_1 \geq \cdots \geq \kappa_n.\]
Then we have
\[\sigma_{k-1}(\kappa) > \kappa_1\cdots \kappa_{k-1}, \quad \sigma_k(\kappa) \leq C(n,k)\ \kappa_1\cdots \kappa_k.\]
\end{lemma}

\begin{lemma} \label{NM}
Let $n \geq k>l \geq 0$ and $n \geq r>s \geq 0$. Suppose $\kappa \in \Gamma_k$. If $k \geq r$ and $l \geq s$, then
\[\left[\frac{\binom{n}{k}^{-1}\sigma_k(\kappa)}{\binom{n}{l}^{-1}\sigma_l(\kappa)}\right]^{\frac{1}{k-l}} \leq \left[\frac{\binom{n}{r}^{-1} \sigma_r(\kappa)}{\binom{n}{s}^{-1}\sigma_s(\kappa)}\right]^{\frac{1}{r-s}}.\] Consequently, for $1 \leq l<k \leq n$, we have
\[k(n-l+1)\sigma_k(\kappa)\sigma_{l-1}(\kappa) \leq l(n-k+1) \sigma_{k-1}(\kappa)\sigma_{l}(\kappa).\]
\end{lemma}
\begin{lemma} \label{concavity 1}
Let $0 \leq l < k \leq n$.
For $\kappa=(\kappa_1,\ldots,\kappa_n) \in \Gamma_k$ and $\xi=(\xi_1,\ldots,\xi_n) \in \bR^n$, we have
\[\sum_{i,j} \left(\frac{\partial^2}{\partial \kappa_i \partial \kappa_j}\left[\frac{\sigma_k(\kappa)}{\sigma_l(\kappa)}\right]\right) \xi_i\xi_j\leq \left(1-\frac{1}{k-l}\right)\frac{\left[\sum_{i}\left(\frac{\partial}{\partial \kappa_i} \left[\frac{\sigma_k(\kappa)}{\sigma_l(\kappa)}\right]\right)\xi_{i}\right]^2}{\left[\frac{\sigma_k(\kappa)}{\sigma_l(\kappa)}\right]}.\] If $\kappa_1 \geq \cdots \geq \kappa_n$, then $\kappa_k>0$ and
\[\frac{\partial}{\partial \kappa_1}\left[\frac{\sigma_k(\kappa)}{\sigma_l(\kappa)}\right] \leq \cdots \leq \frac{\partial}{\partial \kappa_n}\left[\frac{\sigma_k(\kappa)}{\sigma_l(\kappa)}\right].\]
\end{lemma}

For a symmetric matrix $A=(a_{ij})$ and an operator
\[F: \{\text{symmetric matrices}\} \to \bR,\] we define
\begin{equation}
F^{ij}(A)=\frac{\partial F}{\partial a_{ij}}, \quad F^{ij,rs}(A)=\frac{\partial^2 F}{\partial a_{ij} \partial a_{rs}}.\label{derivatives of F}
\end{equation}
When $F(A)=f(\lambda(A))$ depends only on the eigenvalues $\lambda(A)=(\lambda_1,\ldots,\lambda_n)$ of $A$ in the form of a smooth symmetric function $f$, the function $F$ is as smooth as $f$ and is concave if $f$ is concave. Moreover, when $A$ is diagonal, we have $F^{ij}=f_i\delta_{ij}$ where
\begin{equation}
f_i=\frac{\partial f}{\partial \lambda_i}.\label{derivatives of f}
\end{equation} Moreover, we have
\[\sum_{i,j} F^{ij}a_{ij}=\sum_{i=1}^{n} f_i(\lambda(A))\lambda_i, \quad \sum_{i,j,k} F^{ij}a_{ik}a_{jk}=\sum_{i=1}^{n} f_i(\lambda(A))\lambda_{i}^2.\]

In this note, we are considering an equation of the form
\[F(A)=f(\lambda(A))=\psi(x,u),\] where 
\[f(\lambda)=\frac{\sigma_k(\lambda)}{\sigma_{l}(\lambda)}, \quad 0 \leq l<k \leq n,\] and the symmetric matrix $A=(a_{ij})$ is given by 
\begin{equation}
a_{ij}=\frac{1}{w}\left[u_{ij}-\frac{u_iu_ku_{jk}}{w(1+w)}-\frac{u_ju_ku_{ik}}{w(1+w)}+\frac{u_iu_ju_ku_lu_{kl}}{w^2(1+w)^2}\right], \quad w=\sqrt{1+|Du|^2}.\label{the graph curvature matrix}
\end{equation} The eigenvalues of $A$ are the principal curvatures $\kappa[\Sigma]=(\kappa_1,\ldots,\kappa_n)$ of the graph $\Sigma$ of $u$.

\begin{notation}
The $\sigma_k$ operator can be regarded as a function on the space of $n \times n$ symmetric matrices, by writing $\sigma_k(A)$ to mean $\sigma_k(\lambda(A))$. In particular, $\sigma_k(A)$ is the sum of principal minors of order $k \times k$ of $A$.
With this abuse of notation, we may also write $\sigma_{k}^{ij}$ and $\sigma_{k}^{ij,kl}$ to mean the first and the second order derivatives in the sense of \eqref{derivatives of F}. When the matrix is diagonal, the notations $\sigma_{k}^{ii}$ and $\sigma_{k}^{ii,jj}$ can be used to denote the derivatives in the sense of \eqref{derivatives of f} without causing confusion. 
\end{notation}
\begin{lemma} \label{concavity 2}
Let $n \geq 2$ and $1 \leq l<k \leq n$.
For a symmetric matrix $A$, we use $\sigma_{k}(A)$ to mean $\sigma_{k}(\lambda(A))$, where $\lambda(A)=(\lambda_1,\ldots,\lambda_n) \in \Gamma_k$ are the eigenvalues of $A$. Assume $A=a_{ij}=\lambda_i\delta_{ij}$ is diagonal. Then, for $F=\sigma_k/\sigma_l$, we have
\[F^{ij}=\left(\frac{\sigma_{k}^{ij}}{\sigma_l}-\frac{\sigma_k\sigma_{l}^{ij}}{\sigma_{l}^2}\right)\delta_{ij}\] and
\[F^{pq,rs}=\begin{cases}
-2\frac{\sigma_{k}^{pp}\sigma_{l}^{pp}}{\sigma_{l}^2}+2\frac{\sigma_k(\sigma_{l}^{pp})^2}{\sigma_{l}^3}, & p=q=r=s, \\
\frac{\sigma_{k}^{pp,rr}}{\sigma_l}-\frac{\sigma_{k}^{pp}\sigma_{l}^{rr}}{\sigma_{l}^2}-\frac{\sigma_{k}^{rr}\sigma_{l}^{pp}}{\sigma_{l}^2}-\frac{\sigma_k\sigma_{l}^{pp,rr}}{\sigma_{l}^2}+2\frac{\sigma_k\sigma_{l}^{pp}\sigma_{l}^{rr}}{\sigma_{l}^3}, & p=q, r=s, p \neq r \\
\frac{\sigma_{k}^{pq,qp}}{\sigma_{l}}-\frac{\sigma_k\sigma_{l}^{pq,qp}}{\sigma_{l}^2}, & p=s, q=r, p \neq q \\
0, & \text{otherwise}.
\end{cases}\] Moreover, for $p \neq q$, if additionally $\lambda_p \neq \lambda_q$, then
\[-F^{pq,qp}=\frac{F^{pp}-F^{qq}}{\lambda_q-\lambda_p} \geq 0.\]
\end{lemma}
\begin{proof}
The formulas for the first and the second order derivatives can be found in \cite[Lemma 2.2]{LT-CPAA} or \cite[Lemma 2.2]{Lu-2}; they may also be checked by direct computations which we shall omit. Here we only verify the last assertion.

First, for the $\sigma_k$ operator, we can find that
\[\sigma_{k}^{ij}=\sigma_{k-1}(\lambda|i)\delta_{ij},\] and
\[\sigma_{k}^{pq,rs}=\begin{cases}
\sigma_{k-2}(\lambda|pr), & p=q, r=s, p \neq r, \\
-\sigma_{k-2}(\lambda|pr), & p=s, q=r, p \neq q, \\
0, & \text{otherwise}.
\end{cases}\] Next, by Lemma \ref{sigma_k properties 1}, we have the identity
\[\sigma_{k}^{pp}-\sigma_{k}^{qq}=(\lambda_q-\lambda_p) \sigma_{k}^{pp,qq} \quad \text{for all $\lambda \in \bR^n$ and all $1 \leq p,q \leq n$}.\] If $\lambda_p \neq \lambda_q$, then we have
\[\sigma_{k}^{pp,qq}=\frac{\sigma_{k}^{pp}-\sigma_{k}^{qq}}{\lambda_q-\lambda_p}.\] Using these and the formulas for the quotient operator $F$, we get if $p \neq q$ and $\lambda_p \neq \lambda_q$, then
\begin{align*}
-F^{pq,qp}&=-\frac{\sigma_{k}^{pq,qp}}{\sigma_l}+\frac{\sigma_k\sigma_{l}^{pq,qp}}{\sigma_{l}^2} \\
&=\frac{\sigma_{k}^{pp,qq}}{\sigma_{l}}-\frac{\sigma_k\sigma_{l}^{pp,qq}}{\sigma_{l}^2} \\
&=\frac{1}{\sigma_l} \left(\frac{\sigma_{k}^{pp}-\sigma_{k}^{qq}}{\lambda_q-\lambda_p}\right)-\frac{\sigma_{k}}{\sigma_{l}^2}\left(\frac{\sigma_{l}^{pp}-\sigma_{l}^{qq}}{\lambda_q-\lambda_p}\right)\\
&=\frac{1}{\lambda_{q}-\lambda_{p}}\left[\left(\frac{\sigma_{k}^{pp}}{\sigma_{l}}-\frac{\sigma_{k}\sigma_{l}^{pp}}{\sigma_{l}^2}\right)-\left(\frac{\sigma_{k}^{qq}}{\sigma_l}-\frac{\sigma_k\sigma_{l}^{qq}}{\sigma_{l}^2}\right)\right]\\
&=\frac{F^{pp}-F^{qq}}{\lambda_q-\lambda_p}
\end{align*} which is a non-negative quantity by Lemma \ref{concavity 1}.
\end{proof}

Recall the following property for the $\sigma_k$ operator.

\begin{lemma} \label{sigma_k properties 3}
Let $\kappa \in \Gamma_k$ be ordered as $\kappa_1\geq \cdots \geq \kappa_n$. There exists some $\theta=\theta(n,k)>0$ such that for all $1 \leq j \leq k$, we have
\[\sigma_{k}^{jj}(\kappa) \geq \frac{\theta\sigma_{k}(\kappa)}{\kappa_j}.\] If $j \geq k$, then we have
\[\sigma_{k}^{jj}(\kappa) \geq \theta \sigma_{k-1}(\kappa).\]
\end{lemma}
\begin{proof}
The second inequality is well-known and we omit the proof; see \cite[Theorem 1]{Lin-Trudinger}. For the first inequality, it was stated in \cite{Ren-Wang-2} without proof. Here we provide a proof for completeness. Note that, if $\kappa \in \Gamma_k$, then $\kappa_j \geq \kappa_k>0$ for all $1 \leq j \leq k$, by Lemma \ref{concavity 1}.
If $\sigma_k(\kappa|j) \leq 0$, then from Lemma \ref{sigma_k properties 1}, we have
\[\sigma_{k}^{jj}(\kappa)=\frac{\sigma_k(\kappa)-\sigma_k(\kappa|j)}{\kappa_j} \geq \frac{\sigma_k(\kappa)}{\kappa_j}.\] Assume $\sigma_k(\kappa|j)>0$. Then $(\kappa|j) \in \Gamma_k$ by \cite[Lemma 2.1]{Urbas}. It follows from Lemma \ref{sigma_k properties 2} that
\[\sigma_{k}^{jj}(\kappa)=\sigma_{k-1}(\kappa|j)>\frac{\kappa_1\cdots\kappa_{k}}{\kappa_j}\] and there exists some $C(n,k)>0$ such that
\[\sigma_k(\kappa) \leq C(n,k) \kappa_1 \cdots \kappa_k.\] Combining the two inequalities would yield
\[\sigma_{k}^{jj}(\kappa) \geq \frac{\theta \sigma_k(\kappa)}{\kappa_j}\] for 
\[\theta(n,k):=\min\left\{1, \frac{1}{C(n,k)}\right\}.\]
\end{proof}

We now apply this lemma to show a property of the quotient operator which will be used in Section \ref{our theorem 2 proof}.
\begin{lemma} \label{bounds for derivatives}
Let $0 \leq l < k \leq n$ and let $\kappa \in \Gamma_k$ be ordered as $\kappa_1 \geq \cdots \geq \kappa_n$. Denote the quotient operator by $F=\sigma_k/\sigma_l$. There exists some $C_1=C_1(n,k,l)>0$ such that
\[\sum_{i=1}^{n} F^{ii}(\kappa) \leq C_1 \frac{\sigma_{k-1}(\kappa)}{\sigma_l(\kappa)}.\]
Moreover, if $j \geq k$, then there exists some $C_2=C_2(n,k,l)>0$ such that
\[F^{jj}(\kappa) \geq C_2(n,k,l) \sum_{i=1}^{n} F^{ii}(\kappa).\] While if $1 \leq j \leq k-1$ and $l=k-2 \geq 1$, then there exists some $C_3=C_3(n,k,l)>0$ such that
\[F^{jj}(\kappa) \geq C_3(n,k,l) \frac{F(\kappa)}{\kappa_{j}^2} \sum_{i=1}^{n} F^{ii}(\kappa).\] 
\end{lemma}
\begin{proof}
By a direct computation, we have
\begin{align*}
F^{ii}&=\frac{\sigma_{k}^{ii}\sigma_l-\sigma_k\sigma_{l}^{ii}}{\sigma_{l}^2}, \quad \text{for}\quad 0 \leq l<k \leq n.
\end{align*} The first inequality follows by summing over $i$ and invoking Lemma \ref{sigma_k properties 1}.

For the second and the third inequalities, we apply Lemma \ref{sigma_k properties 1} to get the expansions
\begin{align*}
\sigma_l(\kappa)&=\kappa_i\sigma_{l}^{ii}(\kappa)+\sigma_l(\kappa|i),\\
\sigma_{k}(\kappa)&=\kappa_{i}\sigma_{k}^{ii}(\kappa)+\sigma_k(\kappa|i),
\end{align*} for $0 \leq l<k \leq n$, then we substitute them back and obtain that
\begin{align*}
&\ \sigma_{k}^{ii}\sigma_l-\sigma_k\sigma_{l}^{ii}\\
=&\ \sigma_{k}^{ii}[\kappa_i\sigma_{l}^{ii}(\kappa)+\sigma_l(\kappa|i)]-[\kappa_{i}\sigma_{k}^{ii}(\kappa)+\sigma_k(\kappa|i)]\sigma_{l}^{ii} \\
=&\ \sigma_{k}^{ii}(\kappa)\sigma_l(\kappa|i)-\sigma_k(\kappa|i)\sigma_{l}^{ii}(\kappa).
\end{align*} If $\sigma_{k}(\kappa|i) \leq 0$, then
\[\sigma_{k}^{ii}\sigma_l-\sigma_k\sigma_{l}^{ii} \geq \sigma_{k}^{ii}(\kappa)\sigma_{l}(\kappa|i).\] Assume $\sigma_{k}(\kappa|i)>0$. Then $(\kappa|i) \in \Gamma_k$ by \cite[Lemma 2.1]{Urbas} and from Lemma \ref{NM}, we would have
\begin{align*}
&\ \sigma_k(\kappa|i)\sigma_{l}^{ii}(\kappa) \\
=&\ \sigma_k(\kappa|i)\sigma_{l-1}(\kappa|i) \\
\leq &\ \frac{l(n-k+1)}{k(n-l+1)} \sigma_{k-1}(\kappa|i)\sigma_l(\kappa|i)\\
=&\ \frac{l(n-k+1)}{k(n-l+1)} \sigma_{k}^{ii}(\kappa)\sigma_l(\kappa|i).
\end{align*} Thus, in either case, we always have that
\[\sigma_{k}^{ii}\sigma_l-\sigma_k\sigma_{l}^{ii} \geq C(n,k,l) \sigma_{k}^{ii}(\kappa)\sigma_{l}(\kappa|i).\] Now, let $C(n,k,l)$ denote a universal constant depending only on $n,k,l$, whose value may change from line to line but is still denoted by the same symbol. For $j \geq k$, the case when $l=0$ has been proved in Lemma \ref{sigma_k properties 3}. Assume $l \geq 1$, we apply the second inequality in Lemma \ref{sigma_k properties 3} to obtain that
\begin{align*}
F^{jj} &\geq C(n,k,l) \frac{\sigma_{k}^{jj}(\kappa)}{\sigma_l(\kappa)} \cdot \frac{\sigma_{l}(\kappa|j)}{\sigma_{l}(\kappa)}\\
&\geq C(n,k,l) \frac{\sigma_{k-1}(\kappa)}{\sigma_{l}(\kappa)} \cdot \frac{\sigma_{l}(\kappa)}{\sigma_{l}(\kappa)} \\
&\geq C(n,k,l) \sum_{i=1}^{n} F^{ii}.
\end{align*}
Finally, for $1 \leq j \leq k-1$, if $l \geq 1$, we can apply the first inequality in Lemma \ref{sigma_k properties 3} to get
\begin{align*}
F^{jj}&=\frac{\sigma_{k}^{jj}\sigma_l-\sigma_k\sigma_{l}^{jj}}{\sigma_{l}^2} \\
&\geq C(n,k,l) \cdot \frac{\sigma_{k}^{jj}(\kappa)}{\sigma_l(\kappa)} \cdot \frac{\sigma_{l}(\kappa|j)}{\sigma_{l}(\kappa)}\\
&\geq C(n,k,l) \cdot \frac{1}{\kappa_j}\frac{\sigma_{k}(\kappa)}{\sigma_l(\kappa)} \cdot \frac{1}{\kappa_{j}} \frac{\sigma_{l+1}(\kappa)}{\sigma_{l}(\kappa)}
\end{align*} If $l=k-2$, the inequality reduces to
\[F^{jj} \geq \frac{C(n,k,l)}{\kappa_{j}^2} \frac{\sigma_{k}(\kappa)}{\sigma_{k-2}(\kappa)} \frac{\sigma_{k-1}(\kappa)}{\sigma_{k-2}(\kappa)} \geq C(n,k,l)\frac{F(\kappa)}{\kappa_{j}^2}\sum_{i=1}^{n} F^{ii}\] by the first inequality that we just proved.  Note that the restriction $j \leq k-1$ comes from $j \leq l+1=k-1$.

The proof is now complete.
\end{proof}

Finally, we conclude this section with a concavity inequality.
\begin{lemma}\label{concavity 3}
Let $3 \leq k \leq n-1$, $K \geq 0$, and $\psi_0>0$. Set
\[f=\frac{\sigma_k}{\sigma_{k-2}}: \Gamma_k \to (0,\infty).\]
Suppose the components of $\mu \in \Gamma_k$ satisfy
\[\mu_1=\cdots=\mu_m>\mu_{m+1} \geq \cdots \geq \mu_n \geq -K,\] where $1 \leq m \leq n-1$, and 
\[f(\mu) \geq \psi_0.\]
There exist constants $\varepsilon_{*}>0$ and $K_0,\eta_0 \geq 1$ depending only on $n$, $k$, $K$ and $\psi_0$, such that 
if
\[\mu_1 \geq \eta_0 \sqrt{f(\mu)},\] then
for every $0<\varepsilon_0 \leq \varepsilon_{*}$, the inequality
\begin{equation}
-\sum_{p,q} f_{pq}(\mu) \xi_p\xi_q - (1+\varepsilon_0)\frac{f_1(\mu)\xi_{1}^2}{\mu_1} + 2\sum_{i>m} \frac{f_i(\mu)\xi_{i}^2}{\mu_1-\mu_i} \geq -\frac{K_0}{f}\left(\sum_{i=1}^{n} f_i(\mu)\xi_i\right)^2 \label{concavity}
\end{equation} holds for vectors $\xi \in \bR^n$ with
\[\xi_2=\cdots=\xi_m=0.\] When $m=1$, the inequality holds for all vectors $\xi \in \bR^n$ without any restrictions.
Here $f_i(\mu)$ and $f_{pq}(\mu)$ are the first and the second order partial derivatives of $f$ evaluated at $\mu$. The constants are independent of $m$, $\mu$, and $\xi$.
\end{lemma}
\begin{proof}
If $m=1$, then by \cite[Lemma 1.1]{Dong-Zhang}, we have
\[-\sum_{p,q} f_{pq}(\mu) \xi_p\xi_q - (1+\varepsilon_0)\frac{f_1(\mu)\xi_{1}^2}{\mu_1} + \frac{2}{1+\delta_0}\sum_{i>1} \frac{f_i(\mu)\xi_{i}^2}{\mu_1} \geq -\frac{K_0}{f}\left(\sum_{i=1}^{n} f_i(\mu)\xi_i\right)^2\] for a sufficiently small $\delta_0=\delta_0(n,k)>0$. To get our inequality, we first find that
\begin{align*}
&\ -\sum_{p,q} f_{pq}(\mu) \xi_p\xi_q - (1+\varepsilon_0)\frac{f_1(\mu)\xi_{1}^2}{\mu_1}+2\sum_{i>1} \frac{f_i(\mu)\xi_{i}^2}{\mu_1-\mu_i}\\
\geq&\ -\frac{K_0}{f}\left(\sum_{i=1}^{n} f_i(\mu)\xi_i\right)^2 + 2\sum_{i>1} \frac{f_i(\mu)\xi_{i}^2}{\mu_1-\mu_i} - \frac{2}{1+\delta_0}\sum_{i>1} \frac{f_i(\mu)\xi_{i}^2}{\mu_1}
\end{align*} By enlarging $\eta_0$, we can assume that
\[\mu_1 \geq \frac{K}{\delta_0}\] and so
\begin{align*}
&\ 2\sum_{i>1} \frac{f_i(\mu)\xi_{i}^2}{\mu_1-\mu_i} - \frac{2}{1+\delta_0}\sum_{i>1} \frac{f_i(\mu)\xi_{i}^2}{\mu_1} \\
=&\ 2\sum_{i>1} f_i(\mu)\xi_{i}^2 \left(\frac{1}{\mu_1-\mu_i}-\frac{1}{1+\delta_0}\frac{1}{\mu_1}\right)\\
=&\ 2\sum_{i>1} f_i(\mu)\xi_{i}^2 \cdot \frac{(1+\delta_0)\mu_1-(\mu_1-\mu_i)}{(1+\delta_0)\mu_1(\mu_1-\mu_i)}\\
\geq &\ 2\sum_{i>1} f_i(\mu)\xi_{i}^2 \cdot \frac{\delta_0\mu_1-K}{(1+\delta_0)\mu_1(\mu_1-\mu_i)} \\
\geq &\ 0.
\end{align*}
This implies the desired inequality for $m=1$. Now assume $2 \leq m \leq n-1$ and
\[\xi_2=\cdots=\xi_m=0.\] For $\rho>0$, we consider
\[\mu^{\rho}:=(\mu_1,\mu_2-\rho,\ldots,\mu_m-\rho,\mu_{m+1},\ldots,\mu_n).\] For $0<\rho(f,\Gamma_k,\psi_0)<\mu_1-\mu_{m+1}$ sufficiently small, we have 
\[\mu^\rho \in \Gamma_k, \quad \mu_{1}>\mu_{2}^{\rho} \geq \cdots \geq \mu_{m}^{\rho}>\mu_{m+1}^{\rho} \geq \cdots \geq \mu_{n}^{\rho},\]
and 
\[\frac{\psi_0}{2} \leq f(\mu^{\rho}) \leq f(\mu), \quad \frac{\mu_1}{\sqrt{f(\mu^{\rho})}} \geq \frac{\mu_1}{\sqrt{f(\mu)}} \geq \eta_0\] by ellipticity ($f_i>0$ in $\Gamma_k$). Apply the $m=1$ case of the inequality to $\mu^\rho$, we get that
\[-\sum_{p,q} f_{pq}(\mu^\rho) \xi_p\xi_q - (1+\varepsilon_0)\frac{f_1(\mu^\rho)\xi_{1}^2}{\mu_1} + 2\sum_{i>1} \frac{f_i(\mu^\rho)\xi_{i}^2}{\mu_1-\mu_{i}^{\rho}} \geq -\frac{K_0}{f(\mu^\rho)}\left(\sum_{i=1}^{n} f_i(\mu^\rho)\xi_i\right)^2.\] Note that
\[2\sum_{1<i \leq m}\frac{f_i(\mu^\rho)\xi_{i}^2}{\mu_1-\mu_{i}^{\rho}} = \frac{2}{\rho}\sum_{1 < i \leq m} f_i(\mu^\rho)\xi_{i}^2=0.\] We are then left with
\[-\sum_{p,q} f_{pq}(\mu^\rho) \xi_p\xi_q - (1+\varepsilon_0)\frac{f_1(\mu^\rho)\xi_{1}^2}{\mu_1} + 2\sum_{i>m} \frac{f_i(\mu^\rho)\xi_{i}^2}{\mu_1-\mu_{i}} \geq -\frac{K_0}{f(\mu^\rho)}\left(\sum_{i=1}^{n} f_i(\mu^\rho)\xi_i\right)^2.\] The desired inequality for $2 \leq m \leq n-1$ follows by taking $\rho \to 0$.
\end{proof}

\section{Interior curvature estimates} \label{our theorem 2 proof}
In this section, we prove Theorem \ref{our theorem 2}. Before starting, we clarify some notations. The symbol $C$ will be used to denote a universal positive constant whose value will not be specified but will depend only on known constants $n, k, r, \norm{u}_{L^\infty(B_r)}, \norm{Du}_{L^{\infty}(B_r)}$, $(\inf_{\cD} \psi)^{-1}$ and $\norm{\psi}_{C^{1,1}(\cD)}$, where
\[\cD:=\overline{B_{\frac{3r}{4}}} \times [-M, M], \quad M:=1+\norm{u}_{L^{\infty}(B_r)}.\] The magnitude of $C$ will change from line to line and still be denoted by the same symbol. This shall not cause confusion because its exact value is not of importance in our analysis. 

Let $u$ be a $C^4$ function in $B_r$. For $x \in B_r$, we use $X(x)=(x,u(x))$ to denote the corresponding point on the graph $\Sigma$. For a local orthonormal frame $\{e_1,\ldots,e_n\}$ on $\Sigma$, covariant differentiation in the direction $e_i$ is denoted by $\nabla_i$. The components of the second fundamental form of $\Sigma$ with respect to this basis are denoted by $h_{ij}$. Hence,
\[h_{ij}=\langle D_{e_i}e_j,\nu \rangle,\] where $D$ and $\langle \cdot,\cdot \rangle$ are the usual connection and inner product on $\bR^{n+1}$, and 
\[\nu=\frac{( -Du,1)}{\sqrt{1+|Du|^2}}\] is the upward unit normal vector field. For convenience, we shall write $h_{ijk}=\nabla_k h_{ij}$ and $h_{ijkl}=\nabla_{lk} h_{ij}=\nabla_l \nabla_k h_{ij}$, etc, for covariant differentiations.
\begin{proof}[Proof of Theorem \ref{our theorem 2}]
The end-point cases $k=2$ and $k=n$ were treated in \cite{Guan-Qiu} and \cite{Jianxiang}, respectively; hence we assume $3 \leq k \leq n-1$. For simplicity, we may work on the unit ball $B_1$ and the argument can be carried over to a general ball $B_r$ by rescaling. Specifically, given a ball $B_r$, we set $R=3r/4$, and define
\begin{equation}
\widetilde{u}(y)=\frac{u(Ry)}{R},
\qquad
\widetilde{\psi}(y,z)=R^2\psi(Ry,Rz),
\qquad y\in B_{4/3}. \label{scaling}
\end{equation}

Then
\[
D\widetilde{u}(y)=Du(Ry),
\qquad
D^2\widetilde{u}(y)=R D^2u(Ry).
\]

The principal curvatures of the graph of $\widetilde{u}$, with respect to its upward unit normal, satisfy
\[
\widetilde{\kappa}_i(y)=R\kappa_i(Ry).
\]

Since $\sigma_k/\sigma_{k-2}$ is homogeneous of degree two, the function $\widetilde{u} \in C^4(B_{4/3})$ satisfies
\[
\frac{\sigma_k}{\sigma_{k-2}}\bigl(\widetilde{\kappa}(y)\bigr)
=\widetilde{\psi}\bigl(y,\widetilde{u}(y)\bigr)
\quad\text{in }B_{4/3}.
\] Thus, $\widetilde{u} \in C^4(\overline{B_1})$ and the original data cylinder
\[
\mathcal{D}=\overline{B_{3r/4}}\times[-M,M],
\qquad \text{where $M=1+\|u\|_{L^\infty(B_r)}$},
\]
is transformed into
\[
\widetilde{\mathcal{D}}
=\overline{B_1}\times [-M/R,M/R].
\] Note that the norms and positive lower bound of $\widetilde{\psi}$ on this cylinder are controlled by the data in the statement of the theorem.

We can now work on the unit ball $B_1$. Until the final scaling-back step, we suppress the tildes, relabel $y$ as $x$, and use $X$, $\nu$, $h_{ij}$, and $\kappa_i$ for the corresponding quantities of the rescaled graph. Denote
\[
\rho(x)=1-|x|^2,
\qquad E_{n+1}=(0,\ldots,0,1),
\]
and
\[
V(x)=-\beta\frac{\langle X(x),\nu(x)\rangle}{\langle \nu(x), E_{n+1}\rangle}
+\frac{\alpha}{\langle \nu(x),E_{n+1}\rangle^2},\] where $\alpha$, $\beta$ and $\Lambda>1$, are positive constants whose values depend only on the known data. Consider the auxiliary function
\[
P(x)=2\log\rho(x)+\log\widehat{b}(x)+V(x),
\quad x\in B_1,
\] where
\[\widehat{b}(x)=\log\max\{\kappa_{\max}(x),\Lambda\}.\]
Since $u \in C^4(\overline{B_1})$, the functions $\widehat{b}$ and $V$ are continuous and bounded on $\overline{B_1}$. Therefore
\[
Q(x):=e^{P(x)}
=\rho^2(x)\widehat{b}(x)e^{V(x)}
\]
extends continuously to $\overline{B_1}$ by setting $Q=0$ on $\partial B_1$. Since $Q>0$ in $B_1$, it attains its maximum at some interior point $x_0\in B_1$. Equivalently, $P$ attains its global maximum at $x_0$. Moreover,
\[
P(x)\longrightarrow-\infty
\qquad\text{as }x\longrightarrow\partial B_1.
\]

Our aim is to bound
\[
\rho^2(x_0)\widehat b(x_0)
\]
by a constant depending only on the known data. If \(\kappa_{\max}(x_0)\le\Lambda\), then the desired estimate follows. We therefore consider the case \(\kappa_{\max}(x_0)>\Lambda\).

Let $X_0:=(x_0,u(x_0))$.
We choose orthonormal coordinates $\{E_1,\ldots,E_{n+1}\}$ in $\bR^{n+1}$ such that $E_i \perp E_{n+1}$ and so we may decompose the position vector as
\[X=\sum_{i=1}^{n} \langle X,E_i\rangle E_i + \langle X,E_{n+1}\rangle E_{n+1}.\] Thus,
\[\rho=1-\sum_{i=1}^{n} \langle X,E_i\rangle^2.\] We also choose local orthonormal frame $\{e_1,\ldots,e_n\}$ such that the second fundamental form $h_{ij}(x_0)=\kappa_i(x_0)\delta_{ij}$ is diagonal with
\[\kappa_{\max}(x_0)=\kappa_1(x_0) \geq \cdots \geq \kappa_n(x_0),\] and
\[
\nabla_{e_i}e_j(X_0)=0.
\]

Let 
\[F=\frac{\sigma_k}{\sigma_{k-2}}.\]
We shall divide the proof into two parts. In the first part, we show that the concavity inequality \eqref{concavity} will lead to a so-called Jacobi inequality. In the second part, with this Jacobi inequality at hand, we carry out a modified version of Guan-Qiu's delicate analysis \cite{Guan-Qiu} to get the curvature estimate. 

\begin{remark}
The preliminary computation in the first part can be omitted at the first reading, and the reader may directly jump to the second part and see how the inequality is to be used.
\end{remark}

\subsection{Part I: preliminary computations}
If the largest eigenvalue is simple at $x_0$, then $\kappa_1$ is $C^2$ around $x_0$. We begin to compute, at $x_0$,
\begin{gather} \label{the smooth case} 
\begin{split}
\nabla_i \kappa_1&=h_{11i},\\
\nabla_i\nabla_i \kappa_1&=h_{11ii}+2\sum_{p>1}\frac{h_{1pi}^2}{\kappa_1-\kappa_p}.
\end{split}
\end{gather} However, we shall also treat the case when $\kappa_1$ has multiplicity more than $1$, that is, when
\[\kappa_1=\kappa_2=\cdots=\kappa_m>\kappa_{m+1} \geq \cdots \geq \kappa_n \] for some $m>1$ at $x_0$. In this case, let $\varphi$ be a $C^2$ function such that
\[P(X_0)=2\log \rho(X)+\log \log \varphi - \beta \frac{\langle X, \nu \rangle}{\langle \nu, E_{n+1}\rangle}+\alpha \frac{1}{\langle \nu, E_{n+1}\rangle^2},\] where $X_0=(x_0,u(x_0))$. Since $X_0$ is a local maximum of $P$, we have $\varphi \geq \kappa_1$ near $X_0$ and $\varphi = \kappa_1$ at $X_0$. Hence, we can apply a smooth approximation lemma \cite[Lemma 5]{BCD} to obtain that, at $x_0$,
\begin{align}
\delta_{kl} \cdot \nabla_{i} \varphi&=\nabla_i h_{kl}, \quad 1 \leq k,l \leq m, \label{1st order approximation}\\
\nabla_i\nabla_i \varphi & \geq h_{11ii}+2\sum_{p > m} \frac{h_{1pi}^2}{\kappa_1-\kappa_p}. \label{2nd order approximation}
\end{align} Now, take
\begin{equation}
b=\begin{cases}
\log \kappa_1, & m = 1 \\
\log \varphi,  & m > 1
\end{cases}, \label{definition of b}
\end{equation} we are going to derive a Jacobi inequality for the quantity $b$ from the concavity assumption \eqref{concavity}.

\begin{remark}\label{multiplicity n}
If $m=n$, then 
\[\kappa_1(x_0)=\cdots=\kappa_n(x_0).\] From the equation and the fact that $\kappa_i \geq 0$ for all $i$, it would follow that
\[\psi=\frac{\sigma_k}{\sigma_{k-2}}\geq C(n,k) \kappa_k\kappa_{k-1} = C(n,k) \kappa_{1}^2.\] Hence, the desired curvature bound would be obtained immediately. Therefore, we may assume $m \leq n-1$ to match the hypothesis in Lemma \ref{concavity 3}.
\end{remark}

By \eqref{the smooth case}, \eqref{1st order approximation}, and \eqref{2nd order approximation}, we can always find that, for any multiplicity $m \geq 1$, the following is true at $x_0$,
\begin{align*}
b_i&=\frac{h_{11i}}{\kappa_1},\\
b_{ii}&\geq \frac{h_{11ii}}{\kappa_1}+2\sum_{p>m} \frac{h_{1pi}^2}{\kappa_1(\kappa_1-\kappa_p)}-\frac{h_{11i}^2}{\kappa_{1}^2}.
\end{align*}  Contracting with $F^{ii}$, we consider 
\begin{align}
&\ \sum_{i=1}^{n} F^{ii}b_{ii} - \varepsilon \sum_{i=1}^{n} F^{ii}b_{i}^2 \nonumber \\
\geq &\ \sum_{i=1}^{n} \frac{F^{ii}h_{11ii}}{\kappa_1}+2\sum_{i=1}^{n}\sum_{p>m} \frac{F^{ii}h_{1pi}^2}{\kappa_1(\kappa_1-\kappa_p)}-(1+\varepsilon)\sum_{i=1}^{n} \frac{F^{ii}h_{11i}^2}{\kappa_{1}^2}, \label{Jacobi 1}
\end{align} 
where $\varepsilon>0$ is a universal constant whose value is to be determined later. Let us deal with the first term in \eqref{Jacobi 1}. By the commutator formula, the Codazzi equation, and the Gauss equation, we have
\begin{align*}
h_{11ii}&=h_{ii11}+\sum_{m} (h_{1m}h_{m1}h_{ii}-h_{mi}^2h_{11})\\
&=h_{ii11}+\kappa_{1}^2\kappa_i - \kappa_{i}^2 \kappa_1;
\end{align*} see, for example, \cite[expression (39)]{Guan-Qiu}. Therefore, the first term in \eqref{Jacobi 1} reads
\begin{align*}
\sum_{i=1}^{n} \frac{F^{ii}h_{11ii}}{\kappa_1}&=\sum_{i=1}^{n} \frac{F^{ii}h_{ii11}}{\kappa_1}+\kappa_1\sum_{i=1}^{n} F^{ii}\kappa_i-\sum_{i=1}^{n} F^{ii}\kappa_{i}^2 \\
&\geq \sum_{i=1}^{n} \frac{F^{ii}h_{ii11}}{\kappa_1}+\kappa_1\inf_{B_1} \psi-\sum_{i=1}^{n} F^{ii}\kappa_{i}^2,
\end{align*} where we have applied Lemma \ref{sigma_k properties 1} to the quotient $\sigma_k/\sigma_l$ and get
\[\sum_{i=1}^{n} F^{ii}\kappa_i = (k-l)F \geq \inf_{B_1} \psi \quad \text{for $k-l \geq 1$}.\] We continue to expand the $h_{ii11}$ term by differentiating the equation $F(h_{ij})=\psi$ twice, that is,
\begin{equation}
\sum_{i=1}^{n} F^{ii}h_{ii1}=\nabla \psi(e_1) \label{differentiate once}
\end{equation} and
\[\sum_{i=1}^{n} F^{ii}h_{ii11}+\sum_{p,q,r,s} F^{pq,rs}h_{pq1}h_{rs1}=\nabla^2\psi(e_1,e_1) \geq -C-C\kappa_1\] for some $C>0$ depending on $\norm{\psi}_{C^{1,1}}$ and $\sup |Du|$. Putting these expansions altogether back into \eqref{Jacobi 1}, we obtain that
\begin{gather} \label{Jacobi 2}
\begin{split}
&\ \sum_{i=1}^{n} F^{ii}b_{ii} - \varepsilon \sum_{i=1}^{n} F^{ii}b_{i}^2  \\
\geq &\ -\sum_{p,q,r,s} \frac{F^{pq,rs}h_{pq1}h_{rs1}}{\kappa_1}-\sum_{i=1}^{n} F^{ii}\kappa_{i}^2+C\kappa_1-C \\
&\ +2\sum_{i=1}^{n}\sum_{p>m} \frac{F^{ii}h_{1pi}^2}{\kappa_1(\kappa_1-\kappa_p)}-(1+\varepsilon)\sum_{i=1}^{n} \frac{F^{ii}h_{11i}^2}{\kappa_{1}^2}.
\end{split}
\end{gather} We proceed by expanding the remaining terms. By Lemma \ref{concavity 2} and \eqref{1st order approximation}, we have that
\begin{align*}
-\sum_{p,q,r,s} F^{pq,rs} h_{pq1}h_{rs1}&=-\sum_{p,q} F^{pp,qq}h_{pp1}h_{qq1}-\sum_{p \neq q} F^{pq,qp} h_{pq1}^2 \\
&= -\sum_{p,q} F^{pp,qq}h_{pp1}h_{qq1}-2\sum_{i> \max\{j,m\}} F^{ij,ji} h_{ij1}^2 \\
&\geq -\sum_{p,q} F^{pp,qq}h_{pp1}h_{qq1}+2\sum_{i > m} \frac{F^{ii}-F^{11}}{\kappa_1-\kappa_i} h_{11i}^2.\\
\end{align*} On the other hand,
\begin{align*}
2\sum_{i=1}^{n}\sum_{p>m} \frac{F^{ii}h_{1pi}^2}{\kappa_1(\kappa_1-\kappa_p)} &\geq 2\sum_{p>m} \frac{F^{pp}h_{1pp}^2}{\kappa_1(\kappa_1-\kappa_p)}+2\sum_{p>m} \frac{F^{11}h_{1p1}^2}{\kappa_1(\kappa_1-\kappa_p)} \\
&=2\sum_{i>m} \frac{F^{ii}h_{ii1}^2}{\kappa_1(\kappa_1-\kappa_i)} + 2\sum_{i>m} \frac{F^{11}h_{11i}^2}{\kappa_1(\kappa_1-\kappa_i)}.
\end{align*} Moreover, we find that
\begin{equation}
\sum_{i=1}^{n} \frac{F^{ii}h_{11i}^2}{\kappa_{1}^2}=\frac{F^{11}h_{111}^2}{\kappa_{1}^2} + \sum_{i>m} \frac{F^{ii}h_{11i}^2}{\kappa_{1}^2}, \label{third order terms}
\end{equation} because we can deduce from \eqref{1st order approximation} that
\[h_{11i}=h_{1i1}=\delta_{1i} \cdot \nabla_1 \varphi =0 \quad \text{for $1<i \leq m$}.\] Note that the equality in \eqref{third order terms} holds trivially when $m=1$.

Substituting these back into \eqref{Jacobi 2}, we get
\begin{gather} \label{Jacobi 3}
\begin{split}
&\ \sum_{i=1}^{n} F^{ii}b_{ii} - \varepsilon \sum_{i=1}^{n} F^{ii}b_{i}^2  \\
\geq &\ -\sum_{p,q} \frac{F^{pp,qq}h_{pp1}h_{qq1}}{\kappa_1}-(1+\varepsilon)\frac{F^{11}h_{111}^2}{\kappa_{1}^2}+2\sum_{i>m} \frac{F^{ii}h_{ii1}^2}{\kappa_1(\kappa_1-\kappa_i)}\\
&\ + 2\sum_{i>m} \frac{F^{ii}-F^{11}}{\kappa_1(\kappa_1-\kappa_i)} h_{11i}^2 +2\sum_{i>m} \frac{F^{11}h_{11i}^2}{\kappa_1(\kappa_1-\kappa_i)}-(1+\varepsilon)\sum_{i>m} \frac{F^{ii}h_{11i}^2}{\kappa_{1}^2} \\
&\ -\sum_{i=1}^{n} F^{ii}\kappa_{i}^2+C\kappa_1-C. \\
\end{split}
\end{gather} For the second line in \eqref{Jacobi 3}, we apply the concavity inequality \eqref{concavity} to the vector $\xi=(h_{111},h_{221},\ldots,h_{nn1})$ and invoke \eqref{differentiate once},
\begin{gather} \label{k+1 convexity 1}
\begin{split}
&\ -\sum_{p,q} \frac{F^{pp,qq}h_{pp1}h_{qq1}}{\kappa_1}-(1+\varepsilon)\frac{F^{11}h_{111}^2}{\kappa_{1}^2}+2\sum_{i>m} \frac{F^{ii}h_{ii1}^2}{\kappa_1(\kappa_1-\kappa_i)}\\
\geq &\ (1+\varepsilon_0)\frac{F^{11}h_{111}^2}{\kappa_{1}^2}-(1+\varepsilon)\frac{F^{11}h_{111}^2}{\kappa_{1}^2}-\frac{K_0}{F\kappa_1}\left(\sum_{i=1}^{n} F^{ii}h_{ii1}\right)^2 \\
\geq &\ -CK_0
\end{split}
\end{gather} for some $C>0$ depending on the known data, if $\varepsilon\leq \varepsilon_0$. 

\begin{remark}
By \eqref{1st order approximation} and the Codazzi property, we have
\[\xi_i=h_{ii1}=h_{i1i}=\delta_{i1} \cdot \nabla_i \varphi=0 \quad \text{for $1 < i \leq m$}.\] This matches the hypothesis in Lemma \ref{concavity 3}.
\end{remark}

Next, for the third line in \eqref{Jacobi 3}, since $\kappa_i \geq 0$ for all $i$ at all $x \in B_r$, we find that
\begin{gather} \label{k+1 convexity 2}
\begin{split}
&\ 2\sum_{i>m} \frac{F^{ii}-F^{11}}{\kappa_1(\kappa_1-\kappa_i)} h_{11i}^2 +2\sum_{i>m} \frac{F^{11}h_{11i}^2}{\kappa_1(\kappa_1-\kappa_i)}-(1+\varepsilon)\sum_{i>m} \frac{F^{ii}h_{11i}^2}{\kappa_{1}^2} \\
=&\ 2 \sum_{i>m} \frac{F^{ii}h_{11i}^2}{\kappa_1(\kappa_1-\kappa_i)}-(1+\varepsilon)\sum_{i>m} \frac{F^{ii}h_{11i}^2}{\kappa_{1}^2} \\
=&\ \sum_{i>m} \frac{(1-\varepsilon)\kappa_1+(1+\varepsilon)\kappa_i}{\kappa_{1}^2(\kappa_1-\kappa_i)} F^{ii}h_{11i}^2 \\
\geq &\ 0 \quad \text{for $0<\varepsilon \leq 1$}.
\end{split}
\end{gather}

Finally, we conclude that, if we had taken a small enough
\[\varepsilon=\min\{1/2,\varepsilon_0\},\] then with either multiplicity $m=1$ or $m>1$, the quantity $b$ in \eqref{definition of b} always satisfies the following inequality at $X_0$,
\begin{equation}
\sum_{i=1}^{n} F^{ii} b_{ii} \geq \varepsilon \sum_{i=1}^{n} F^{ii}b_{i}^2 + C\kappa_1 - \sum_{i=1}^{n} F^{ii}\kappa_{i}^2-CK_0-C. \label{The Jacobi inequality}
\end{equation}
The first part of the proof is now complete.
\begin{remark}
Throughout the derivation, we had also assumed 
\[\kappa_1 (x_0) \geq \max\{\Lambda, \eta_0\sqrt{\sup \psi}\},\] where $\eta_0$ is the constant in the concavity inequality \eqref{concavity} and $\Lambda>0$ is some possibly large constant whose value depends only on the known data. We will continue to assume this without indication in the second part of the proof.
\end{remark}
\subsection{Part II: the maximum principle arguments} \label{part 2 of the proof}
Now, with the inequality \eqref{The Jacobi inequality}, we show how to proceed for deriving interior curvature estimates. Recall that the auxiliary function
\[P(X)=2\log \rho(X)+\log \log \kappa_1(X) - \beta \frac{\langle X,\nu \rangle}{\langle \nu,E_{n+1}\rangle} + \alpha \frac{1}{\langle \nu, E_{n+1} \rangle^2}\] attains a local maximum at the point $X_0$. When the multiplicity $m>1$, we consider the smooth function $\varphi(X)$ with
\[P(X_0)=2\log \rho(X)+\log \log \varphi(X)- \beta \frac{\langle X,\nu \rangle}{\langle \nu,E_{n+1}\rangle} + \alpha \frac{1}{\langle \nu, E_{n+1} \rangle^2}.\] For any multiplicity $m \geq 1$, we always have the following at $X_0$,
\begin{align}
0&=\frac{2\rho_i}{\rho}+\frac{b_i}{b}-\beta \left[\frac{\langle X,\nu \rangle}{\langle \nu, E_{n+1}\rangle}\right]_{i}-\alpha \frac{2\langle \nu_i,E_{n+1}\rangle}{\langle \nu, E_{n+1}\rangle^3},  \label{Jacobi 1st critical} \\
0&\geq 2\frac{\sum_{i=1}^{n} F^{ii}\rho_{ii}}{\rho}-2\frac{\sum_{i=1}^{n} F^{ii}\rho_{i}^2}{\rho^2}+\sum_{i=1}^{n} \frac{F^{ii}b_{ii}}{b}-\sum_{i=1}^{n} \frac{F^{ii}b_{i}^2}{b^2} \nonumber \\
&\quad -\beta \sum_{i=1}^{n} F^{ii} \left[\frac{\langle X,\nu \rangle}{\langle \nu,E_{n+1}\rangle} \right]_{ii}-2\alpha\frac{\sum_{i=1}^{n} F^{ii} \langle \nu_{ii}, E_{n+1}\rangle}{\langle \nu, E_{n+1}\rangle^3} \nonumber \\
&\quad +6\alpha \frac{\sum_{i=1}^{n} F^{ii}\langle \nu_i, E_{n+1}\rangle^2}{\langle \nu,E_{n+1}\rangle^4}. \label{Jacobi 2nd critical 1}
\end{align} 

By standard computations \cite[expressions (43)-(57)]{Guan-Qiu}, we can derive from \eqref{Jacobi 2nd critical 1} that
\begin{gather} \label{Jacobi 2nd critical 2}
\begin{split}
0 \geq &\ -4\frac{\sum_{i=1}^{n} F^{ii}}{\rho} -16 \frac{\sum_{i=1}^{n} F^{ii} a_{i}^2}{\rho^2} \\
&\ +\sum_{i=1}^{n} \frac{F^{ii}b_{ii}}{b}-\sum_{i=1}^{n} \frac{F^{ii}b_{i}^2}{b^2} \\
&\ +C\alpha \sum_{i=1}^{n} F^{ii}\kappa_{i}^2 - C\left(\frac{1}{\rho}+\beta+\alpha\right),
\end{split}
\end{gather} where
\begin{equation}
a_j=\langle X,e_j\rangle - \langle X,E_{n+1} \rangle \langle e_j,E_{n+1}\rangle=\sum_{i=1}^{n} \langle X,E_i\rangle \langle E_i,e_j\rangle. \label{definition of a}
\end{equation}  For the second line in \eqref{Jacobi 2nd critical 2}, we invoke the inequality \eqref{The Jacobi inequality} that we derived in Part I to get
\begin{align*}
&\ \sum_{i=1}^{n} \frac{F^{ii}b_{ii}}{b}-\sum_{i=1}^{n} \frac{F^{ii}b_{i}^2}{b^2} \\
\geq &\ \varepsilon\sum_{i=1}^{n} \frac{F^{ii}b_{i}^2}{b} - \sum_{i=1}^{n} \frac{F^{ii}b_{i}^2}{b^2}+ \frac{C\kappa_1}{b} - \sum_{i=1}^{n} \frac{F^{ii}\kappa_{i}^2}{b}-\frac{CK_0+C}{b}.
\end{align*} Therefore, by assuming $\rho^2(x_0)\log \kappa_1(x_0) \geq C(\beta,\alpha, \varepsilon)$ large, it follows from \eqref{Jacobi 2nd critical 2} that
\begin{gather} \label{Jacobi 2nd critical 3}
\begin{split}
0 \geq &\ \frac{\varepsilon}{2}\sum_{i=1}^{n} \frac{F^{ii}b_{i}^2}{b} + \frac{C\kappa_1}{b}+C\alpha \sum_{i=1}^{n} F^{ii}\kappa_{i}^2  -4\frac{\sum_{i=1}^{n} F^{ii}}{\rho} -16 \frac{\sum_{i=1}^{n} F^{ii} a_{i}^2}{\rho^2}.
\end{split}
\end{gather}

\begin{remark}
How \eqref{Jacobi 2nd critical 1} leads to \eqref{Jacobi 2nd critical 2} is not essential; it is through standard but tedious computations which we better omit in here. The key is to show how to handle the terms
\[-4\frac{\sum_{i=1}^{n} F^{ii}}{\rho} -16 \frac{\sum_{i=1}^{n} F^{ii} a_{i}^2}{\rho^2}\] in \eqref{Jacobi 2nd critical 3} via the case-by-case analysis below.
\end{remark}

Now we conduct the aforementioned case-by-case analysis which was first used by Guan-Qiu \cite{Guan-Qiu} for dealing with the $\sigma_2$ operator and later modified by Liu \cite{Jianxiang} for the $\sigma_n/\sigma_{n-2}$ operator. Here we extend the analysis to work for the $\sigma_k/\sigma_{k-2}$ operator. The major changes are for the Case 2 and 3 below, especially for getting \eqref{claim} which is very tricky; see Remark \ref{concluding remark 2}.

Since our estimate can depend on $\sup_{B_1}|Du|$, there exist controlled, positive constants $v_0$ and $\eta$ such that
\begin{equation}
v_0 \leq \langle \nu, E_{n+1} \rangle \leq 1 \label{bound for nu}
\end{equation} and
\begin{equation}
\frac{1}{\eta}\sum_{i=1}^{n} \langle X,E_i\rangle^2 \leq \sum_{i=1}^{n} a_{i}^2 \leq \eta \sum_{i=1}^{n} \langle X,E_i \rangle^2.\label{bound for a 1}
\end{equation}
\subsection*{Case 1: $\sum_{i=1}^{n} \langle X, E_i \rangle^2 \leq \frac{1}{2}$.}

\indent

Recall that
\[\rho=1-\sum_{i=1}^{n} \langle X,E_i \rangle^2.\]
In this case, we have $\frac{1}{\rho} \leq 2$ and $\sum_{i=1}^{n} a_{i}^2 \leq \eta/2$. By choosing $\alpha>0$ large enough, we would have
\begin{align*}
&\ C\alpha \sum_{i=1}^{n} F^{ii} \kappa_{i}^2 - 4\frac{\sum_{i=1}^{n} F^{ii}}{\rho} -16 \frac{\sum_{i=1}^{n} F^{ii} a_{i}^2}{\rho^2}\\
\geq &\ C\alpha F^{22}\kappa_{2}^2 - C\eta \sum_{i=1}^{n} F^{ii} \\
\geq &\ C\alpha \inf_{B_1} \psi \sum_{i=1}^{n} F^{ii} - C\eta\sum_{i=1}^{n} F^{ii} \\
\geq &\ 0
\end{align*} by Lemma \ref{bounds for derivatives}. The desired estimate would then follow from \eqref{Jacobi 2nd critical 3}.

Therefore, we assume 
\begin{equation}
\sum_{i=1}^{n} \langle X, E_i \rangle^2 \geq \frac{1}{2} \label{bound for a 2}
\end{equation} and hence $|a_j| > 1/\sqrt{2n\eta}$ for some $1 \leq j \leq n$.

\subsection*{Case 2: $j \geq k$.}

\indent

Suppose $\kappa_k \leq C\delta$ for some small $\delta \in (0,1)$ to be determined. Since $\kappa \in \Gamma_k$, we have \cite[(2.4)]{Urbas}
\begin{equation}
\kappa_k + \cdots + \kappa_n > 0, \label{sum of n-k+1}
\end{equation} from which it follows that 
\[|\kappa_j| \leq C(n,k)\kappa_k \leq C(n,k)\delta \quad \text{for $j \geq k$}.\] Recall from the first order critical condition \eqref{Jacobi 1st critical}, we have
\begin{gather} \label{Jacobi 1st critical 2}
\begin{split}
\frac{b_j}{b}&=\frac{4a_j}{\rho}+\beta \left[\frac{\langle X,\nu \rangle}{\langle \nu, E_{n+1}\rangle} \right]_{j} + \alpha \frac{2\langle \nu_{j}, E_{n+1}\rangle}{\langle \nu,E_{n+1}\rangle^3} \\
&=\frac{4a_j}{\rho}-\beta \frac{d_jh_{jj}}{\langle \nu, E_{n+1}\rangle^2}-\alpha \frac{2h_{jj}\langle e_j, E_{n+1}\rangle}{\langle \nu, E_{n+1}\rangle^3}, 
\end{split}
\end{gather} where
\begin{equation}
d_j=\langle X,e_j\rangle \langle \nu, E_{n+1}\rangle - \langle e_j,E_{n+1}\rangle \langle X,\nu \rangle.\label{definition of d}
\end{equation} Hence, we have
\begin{gather} \label{even smaller}
\begin{split}
\bigg|\frac{b_j}{b}\bigg|& \geq \frac{C}{\rho} - C(\beta+\alpha) |\kappa_{j}| \\
& \geq \frac{C}{\rho} - C(\beta+\alpha) \delta \\
& \geq \frac{C}{\rho}
\end{split}
\end{gather} by choosing $\delta$ small enough. By Lemma \ref{bounds for derivatives}, we then have
\begin{align*}
&\ \frac{\varepsilon}{2}\sum_{i=1}^{n} F^{ii}\frac{b_{i}^2}{b}-4\frac{\sum_{i=1}^{n} F^{ii}}{\rho} -16 \frac{\sum_{i=1}^{n} F^{ii} a_{i}^2}{\rho^2}\\
\geq &\ \frac{\varepsilon}{2} F^{jj} \frac{b_{j}^2}{b^2} \cdot b- C\frac{\sum_{i=1}^{n} F^{ii}}{\rho^2} \\
\geq &\ C(n,k) \varepsilon \sum_{i=1}^{n} F^{ii} \cdot \frac{C}{\rho^2} \cdot b - C\frac{\sum_{i=1}^{n} F^{ii}}{\rho^2} \\
\geq &\ 0
\end{align*} by assuming $b \geq \rho^2\log \kappa_1$ is sufficiently large. The estimate for $\kappa_1$ would then follow from \eqref{Jacobi 2nd critical 3}.

While if $\kappa_k \geq C\delta$, then by Lemma \ref{sigma_k properties 2}, Lemma \ref{bounds for derivatives} and the assumption that $\kappa_i \geq 0$ for all $i$, we have
\begin{equation}
F=\frac{\sigma_k}{\sigma_{k-2}} \geq C(n,k)\kappa_{k-1}\kappa_{k} \geq C(n,k)\delta \kappa_{k-1} \label{k+1 convexity 3}
\end{equation} and so
\[\sum_{i=1}^{n} F^{ii} \leq C(n,k) \frac{\sigma_{k-1}}{\sigma_{k-2}} \leq C(n,k)\kappa_{k-1} \leq \frac{C(n,k)\sup \psi}{\delta}.\] The desired estimate would follow from \eqref{Jacobi 2nd critical 3} by noting that in this case
\[\frac{C\kappa_1}{b} - C\frac{\sum_{i=1}^{n} F^{ii}}{\rho^2} \geq \frac{C\kappa_1}{b} - \frac{C}{\delta} \frac{1}{\rho^2} \geq 0\] because
\[\rho^2(x_0)\log \kappa_1(x_0) \quad \text{large} \quad \text{implies} \quad \frac{\rho^2(x_0)\kappa_{1}(x_0)}{b(x_0)} \quad \text{large}.\]

In either situation, we can always handle the terms
\[-4\frac{\sum_{i=1}^{n} F^{ii}}{\rho} -16 \frac{\sum_{i=1}^{n} F^{ii} a_{i}^2}{\rho^2}\] and have the desired bound for $\rho^2(x_0)\log\kappa_1(x_0)$ from \eqref{Jacobi 2nd critical 3}.

\subsection*{Case 3: $1 \leq j \leq k-1$.}

\indent

In this case, we claim that there exists some $1 \leq l_0 \leq k-1$, which may not be the same as $j$, such that
\begin{equation}
|d_{l_0}| \geq \frac{\langle \nu, E_{n+1}\rangle}{4\sqrt{\eta(k-1)}}, \label{claim}
\end{equation} where $\eta$ is the universal constant in \eqref{bound for a 1}. For now, we use this claim to get the curvature estimate and defer the proof of this claim to the end.

If $\kappa_{k-1} \geq C\sqrt{\log \kappa_1}$, then $\kappa_l \geq C \sqrt{\log \kappa_1}$ for all $1 \leq l \leq k-1$. From the first order critical condition \eqref{Jacobi 1st critical 2}, we have
\begin{align*}
\bigg|\frac{b_{l_0}}{b}\bigg|&\geq \bigg|\beta \frac{d_{l_0} h_{l_0l_0}}{\langle \nu,E_{n+1}\rangle^2}+\alpha\frac{2h_{l_0l_0}\langle e_{l_0},E_{n+1}\rangle}{\langle \nu,E_{n+1}\rangle^3}\bigg|-\bigg|\frac{4a_{l_0}}{\rho}\bigg| \\
&\geq \frac{\kappa_{l_0}}{\langle \nu,E_{n+1}\rangle^2} \left( \beta |d_{l_0}|-\frac{2\alpha}{\langle \nu,E_{n+1}\rangle}\right)-\bigg|\frac{4a_{l_0}}{\rho}\bigg| \\
&\geq \frac{\kappa_{l_0}}{\langle \nu,E_{n+1}\rangle^2} \frac{\beta |d_{l_0}|}{2} - \bigg|\frac{4a_{l_0}}{\rho}\bigg| \\
&\geq \frac{\kappa_{l_0}}{\langle \nu,E_{n+1}\rangle} \frac{\beta}{8\sqrt{\eta(k-1)}}-\frac{4\sqrt{\eta}}{\rho}\\
&\geq \frac{\kappa_{l_0}}{\langle \nu,E_{n+1}\rangle} \frac{\beta}{16\sqrt{\eta(k-1)}} \\
&\geq \frac{\beta}{16\sqrt{\eta(k-1)}} \kappa_{l_0},
\end{align*} where in the second inequality we used \eqref{norm of E} to get
\[|\langle e_{l_0},E_{n+1} \rangle| \leq \sqrt{\sum_{i=1}^{n} \langle e_i,E_{n+1} \rangle^2} \leq 1;\] in the third inequality we had chosen
\[\frac{\beta v_{0}^2}{16\sqrt{\eta(k-1)}} \geq \alpha, \quad \text{where $v_0$ is the universal constant in \eqref{bound for nu},}\] so that
\[\beta |d_{l_0}|-\frac{2\alpha}{\langle \nu,E_{n+1}\rangle} \geq \frac{\beta |d_{l_0}|}{2};\] in the fourth inequality we used the claim \eqref{claim}; in the fifth inequality we had assumed
\[\rho\sqrt{\log \kappa_1} \geq \frac{64\eta \sqrt{k-1}}{C\beta}\] so that
\[\frac{\kappa_{l_0}}{\langle \nu,E_{n+1}\rangle} \frac{\beta}{16\sqrt{\eta(k-1)}}-\frac{4\sqrt{\eta}}{\rho} \geq 
\frac{C\beta \sqrt{\log \kappa_1}}{16\sqrt{\eta(k-1)}}- \frac{4\sqrt{\eta}}{\rho} \geq 0\] under the assumption that $\kappa_l \geq C\sqrt{\log \kappa_1}$ for all $1 \leq l \leq k-1$; and in the sixth inequality we used $\langle \nu,E_{n+1} \rangle \leq 1$ to get 
\[\frac{1}{\langle \nu,E_{n+1}\rangle} \geq 1.\]
Thus, by Lemma \ref{bounds for derivatives}, we have
\begin{align*}
&\ \frac{\varepsilon}{2}\sum_{i=1}^{n} F^{ii}\frac{b_{i}^2}{b}- C\frac{\sum_{i=1}^{n} F^{ii}}{\rho^2}\\
\geq &\ \frac{\varepsilon}{2} F^{l_0l_0} \cdot \frac{b_{l_0}^2}{b^2} \cdot b - C \frac{\sum_{i=1}^{n} F^{ii}}{\rho^2} \\
\geq &\ C\varepsilon \cdot \left(\frac{1}{\kappa_{l_0}^2}\sum_{i=1}^{n} F^{ii}\right) \cdot \frac{\beta^2}{256\eta(k-1)} \kappa_{l_0}^2 \cdot b - C\frac{\sum_{i=1}^{n} F^{ii}}{\rho^2} \\
\geq &\ C\varepsilon\beta^2 b\cdot \sum_{i=1}^{n} F^{ii} - C\frac{\sum_{i=1}^{n} F^{ii}}{\rho^2}\\
\geq &\ 0
\end{align*} if $C\varepsilon \beta^2 b \geq C/\rho^2$.

While if $\kappa_{k-1} \leq C\sqrt{\log \kappa_1}$, then by Lemma \ref{bounds for derivatives} and Lemma \ref{sigma_k properties 2}, we have
\begin{align*}
&\ \frac{C\kappa_1}{b} - C \frac{\sum_{i=1}^{n} F^{ii}}{\rho^2} \\
\geq &\ \frac{C \kappa_1}{b} - \frac{C\kappa_{k-1}}{\rho^2} \\
\geq &\ \frac{C \kappa_1}{b} - \frac{C\sqrt{\log \kappa_1}}{\rho^2}\\
\geq &\ 0
\end{align*} by noting that
\[\rho^2(x_0)\log \kappa_1(x_0) \quad \text{large} \quad \text{implies} \quad \frac{\rho^2(x_0)\kappa_{1}(x_0)}{b^{3/2}(x_0)} \quad \text{large}.\] In either case, we can deduce the desired estimate from \eqref{Jacobi 2nd critical 3}.

We now prove \eqref{claim}. Let $\delta_0 \in (0,1)$ be a sufficiently small number. Suppose there exists some $j_0 \geq k$ with $|a_{j_0}| \geq \delta_0$. Then we may proceed as in Case 2 and the curvature estimate follows by taking an even smaller $\delta \leq C\delta_0/(\beta+\alpha)$ in \eqref{even smaller}. Hence, assume
\begin{equation}
|a_{l}| < \delta_0 \quad \text{for all $l \geq k$}.\label{bound for a 3}
\end{equation} Denote the set of indices by
\[I=\{1,\ldots,k-1\} \quad \text{and} \quad J=\{k,\ldots,n\}.\] By \eqref{bound for a 1}, \eqref{bound for a 2}, and \eqref{bound for a 3}, we have
\begin{equation}
|a_{I}|=\sqrt{\sum_{l \in I} a_{l}^2} = \sqrt{\sum_{l=1}^{n} a_{l}^2 - \sum_{l \in J} a_{l}^2}\geq \sqrt{\frac{1}{2\eta}-(n-k+1)\delta_{0}^2} \geq \frac{1}{2\sqrt{\eta}} \label{bound for a 4}
\end{equation} for a sufficiently small $\delta_0=\delta_0(n,k,\eta)$ with
\[(n-k+1)\delta_{0}^2 \leq \frac{1}{4\eta},\] where $\eta$ is the universal constant in \eqref{bound for a 1}. The claim \eqref{claim} would readily follow if we can show that
\[|d_{I}|=\sqrt{\sum_{l \in I} d_{l}^2} \geq \frac{1}{4\sqrt{\eta}}\langle \nu,E_{n+1}\rangle.\] To show this, let us first express each $d_i$ in terms of the $a_i$'s. Recall from \eqref{definition of a} and \eqref{definition of d}, we have
\[d_i=\langle X,e_i \rangle \langle \nu,E_{n+1} \rangle - \langle e_i,E_{n+1} \rangle \langle X,\nu \rangle\] and
\begin{equation}
\langle X,e_i \rangle=a_i + \langle X,E_{n+1} \rangle \langle e_i,E_{n+1} \rangle.\label{definition of a 2}
\end{equation} It follows that
\begin{align}
d_i&=\langle X,e_i \rangle \langle \nu,E_{n+1} \rangle - \langle e_i,E_{n+1} \rangle \langle X,\nu \rangle \nonumber\\
&=[a_i + \langle X,E_{n+1} \rangle \langle e_i,E_{n+1} \rangle] \langle \nu,E_{n+1} \rangle - \langle e_i,E_{n+1} \rangle \langle X,\nu \rangle \nonumber \\
&=a_i \langle \nu, E_{n+1} \rangle + \langle e_i,E_{n+1} \rangle \left[ \langle X,E_{n+1} \rangle \langle \nu, E_{n+1} \rangle - \langle X,\nu \rangle \right]. \label{definition of d 2}
\end{align} Since $E_{n+1}$ is a unit vector and $\{e_1,\ldots,e_n,\nu\}$ is an orthonormal frame of $\bR^{n+1}$, we may write
\begin{equation}
E_{n+1}=\sum_{i=1}^{n} \langle e_i,E_{n+1} \rangle e_i + \langle \nu,E_{n+1} \rangle \nu \label{decompose E}
\end{equation} and note that
\begin{equation}
\sum_{i=1}^{n} \langle e_i,E_{n+1} \rangle^2 + \langle \nu,E_{n+1} \rangle^2 = 1. \label{norm of E}
\end{equation}
Taking the inner product of $X$ and $E_{n+1}$, we get from \eqref{decompose E} that
\begin{equation}
\langle X, E_{n+1} \rangle = \sum_{i=1}^{n} \langle e_i,E_{n+1} \rangle \langle X,e_i \rangle + \langle \nu, E_{n+1} \rangle \langle X,\nu \rangle.\label{inner product of X and E}
\end{equation} Multiplying \eqref{definition of a 2} by $\langle e_i,E_{n+1}\rangle$ and summing over $i$,
\[\sum_{i=1}^{n} \langle e_i,E_{n+1} \rangle a_i = \sum_{i=1}^{n} \langle e_i,E_{n+1} \rangle \langle X,e_i \rangle - \langle X,E_{n+1} \rangle \langle e_i,E_{n+1} \rangle^2.\] Inserting \eqref{inner product of X and E} and \eqref{norm of E} into this equality, we get
\begin{align*}
&\ \sum_{i=1}^{n} \langle e_i,E_{n+1} \rangle a_i\\
=&\ \sum_{i=1}^{n} \langle e_i,E_{n+1} \rangle \langle X,e_i \rangle - \langle X,E_{n+1} \rangle \langle e_i,E_{n+1} \rangle^2 \\
=&\ \langle X,E_{n+1} \rangle - \langle \nu,E_{n+1} \rangle \langle X,\nu \rangle - \langle X,E_{n+1} \rangle  (1-\langle \nu,E_{n+1} \rangle^2) \\
=&\ \langle \nu,E_{n+1} \rangle^2 \langle X,E_{n+1} \rangle - \langle \nu,E_{n+1} \rangle \langle X,\nu \rangle.
\end{align*} Dividing this by $\langle \nu,E_{n+1}\rangle$ and substituting back into \eqref{definition of d 2}, we get
\begin{align}
d_i&=a_i \langle \nu, E_{n+1} \rangle + \langle e_i,E_{n+1} \rangle \left[ \langle X,E_{n+1} \rangle \langle \nu, E_{n+1} \rangle - \langle X,\nu \rangle \right] \nonumber\\
&=a_i \langle \nu, E_{n+1} \rangle + \frac{\langle e_i,E_{n+1} \rangle}{\langle \nu, E_{n+1}\rangle} \sum_{k=1}^{n} \langle e_k,E_{n+1} \rangle a_k.\label{definition of d 3}
\end{align} We can now show a lower bound for $|d_I|$ using bounds for the $a_i$'s. Indeed, multiplying \eqref{definition of d 3} by $a_i$ and summing over $i \in I$, we get
\begin{align*}
&\ \sum_{i \in I} a_i d_i \\
=&\ \langle \nu,E_{n+1} \rangle \sum_{i \in I} a_{i}^2+\frac{\sum_{i \in I} \langle e_i,E_{n+1}\rangle a_i}{\langle \nu,E_{n+1} \rangle} \left(\sum_{k \in I} \langle e_k,E_{n+1} \rangle a_k + \sum_{k \in J} \langle e_k,E_{n+1}\rangle a_k \right) \\
\geq &\ \langle \nu,E_{n+1} \rangle |a_{I}|^2 + \frac{1}{\langle \nu,E_{n+1}\rangle}\left( \sum_{i \in I} \langle e_i,E_{n+1}\rangle a_i\right)\left(\sum_{k \in J} \langle e_k,E_{n+1}\rangle a_k\right). \\
\end{align*} In order to get a lower bound, we estimate the product of the partial sums. By the Cauchy-Schwarz inequality, we have
\begin{align*}
&\ \bigg| \sum_{i \in I} \langle e_i,E_{n+1}\rangle a_i\bigg|\cdot \bigg|\sum_{k \in J} \langle e_k,E_{n+1}\rangle a_k\bigg| \\
\leq &\ \left(\sum_{i \in I} \langle e_i,E_{n+1}\rangle^2\right)^{1/2} \left(\sum_{i \in I} a_{i}^2\right)^{1/2}\left(\sum_{k \in J} \langle e_k,E_{n+1}\rangle^2\right)^{1/2} \left(\sum_{k \in J} a_{k}^2\right)^{1/2}\\
\leq &\ \frac{1}{2}\left(\sum_{i \in I} \langle e_i,E_{n+1}\rangle^2 + \sum_{k \in J} \langle e_k,E_{n+1}\rangle^2\right) \cdot |a_{I}|\cdot |a_{J}| \\
=&\ \frac{|a_I| \cdot |a_J|}{2} \sum_{k=1}^{n} \langle e_k, E_{n+1} \rangle^2.
\end{align*}
By \eqref{norm of E}, we have that
\[\sum_{k=1}^{n} \langle e_k, E_{n+1} \rangle^2=1-\langle \nu,E_{n+1} \rangle^2 \leq 1.\] Thus, we can conclude that
\begin{align*}
\sum_{i \in I} a_id_i \geq |a_I| \left(\langle \nu,E_{n+1} \rangle |a_I|-\frac{|a_J|}{2\langle \nu,E_{n+1}\rangle}\right).
\end{align*} From \eqref{bound for a 3} and \eqref{bound for a 4}, we can choose $\delta_0=\delta_0(n,k,\eta,v_0)>0$ with
\[\frac{v_{0}^2}{2\sqrt{\eta}} \geq (n-k+1)\delta_0,\] where $v_0$ and $\eta$ are the universal constants in \eqref{bound for nu} and \eqref{bound for a 1}, so that
\[\langle \nu,E_{n+1} \rangle |a_I|-\frac{|a_J|}{2\langle \nu,E_{n+1}\rangle} \geq \frac{1}{2}\langle \nu,E_{n+1}\rangle |a_I|\] and hence
\[\sum_{i \in I} a_id_i \geq \frac{1}{2} \langle \nu,E_{n+1} \rangle |a_{I}|^2.\] Finally, by the Cauchy-Schwartz inequality, we have
\[\sum_{i \in I} a_i d_i \leq \left(\sum_{i \in I} a_{i}^2\right)^{1/2}\left(\sum_{i \in I} d_{i}^2\right)^{1/2} = |a_I| \cdot |d_I|.\] Dividing by the common factor, we obtain
\[|d_I| \geq \frac{1}{2} \langle \nu,E_{n+1}\rangle |a_I| \geq \frac{1}{4\sqrt{\eta}}\langle \nu,E_{n+1}\rangle\] and the claim \eqref{claim} readily follows if we had chosen
\[\delta_0 \leq \min \left\{\frac{1}{2\sqrt{\eta(n-k+1)}}, \frac{v_{0}^2}{2(n-k+1)\sqrt{\eta}}\right\}.\] 

Now, after we have handled all the three cases, since $x_0$ is the maximum point of $e^P$, we get that
\[\rho^2(x)\log \max\{\kappa_1(x),\Lambda\} \leq C e^{\osc_{B_1} V},\] where the oscillation of
\[V(x)=-\beta \frac{\langle X(x),\nu(x)\rangle}{\langle\nu(x),E_{n+1}\rangle}+\alpha\frac{1}{\langle \nu(x), E_{n+1}\rangle^2}\] is controlled by the known data. Hence,
\[\sup_{B_{\theta}}\kappa_1 \leq \exp\left(\frac{Ce^{\osc V}}{(1-\theta^2)^2}\right), \quad \theta \in (0,1).\] Applying this estimate with $\theta=2/3$ and scaling back by \eqref{scaling}, we have
\[\sup_{x \in B_{\frac{r}{2}}} \kappa_1(x) \leq \frac{1}{R} \sup_{y \in B_{\frac{2}{3}}} \widetilde{\kappa}_1(y) \leq \frac{C}{r}.\]

The proof is now complete.
\end{proof}

\begin{remark} \label{concluding remark 1}
The proof remains valid for $(k+1)$-admissible solutions as we now show. The places where we used the assumption that $\kappa_i \geq 0$ for all $i$ are Remark \ref{multiplicity n}, \eqref{k+1 convexity 1}, \eqref{k+1 convexity 2} and \eqref{k+1 convexity 3}. Indeed, assume $\kappa \in \Gamma_{k+1}$, then we can still apply Lemma \ref{sigma_k properties 2} to get Remark \ref{multiplicity n} and \eqref{k+1 convexity 3}. Moreover, from \eqref{k+1 convexity 3}, it also follows that
\[\sup \psi \geq C(n,k) \kappa_{k}^2.\] Due to \eqref{sum of n-k+1} which holds for $\kappa \in \Gamma_k$, we have
\[|\kappa_n| \leq C(n,k) \kappa_k \leq C(n,k,\sup \psi).\] This gives us a lower bound
\begin{equation}
\kappa_n \geq -C.\label{the lower bound}
\end{equation} With this lower bound, the inequality \eqref{k+1 convexity 2} still holds with $0<\varepsilon<1$ and a sufficiently large $\kappa_1$. Again with this lower bound, Lemma \ref{concavity 3} can still be applied to obtain the inequality \eqref{k+1 convexity 1}. Hence, we conclude that our interior curvature estimates \eqref{the estimate} remain valid for $(k+1)$-admissible solutions.

\begin{theorem} 
Let $3 \leq k \leq n-1$ and let $r>0$. For a given positive function $\psi(x,u) \in C^{1,1}(\overline{B_r} \times \bR)$, suppose $u \in C^4(B_r)$ is a $(k+1)$-admissible solution of the equation
\[H(D^2u,Du):=\left( \frac{\sigma_k}{\sigma_{k-2}} \right)(\kappa[u])=\psi(x,u(x)) \quad \text{in $B_r$},\] where $\kappa[u]=(\kappa_1,\ldots,\kappa_n)$ are the principal curvatures of its graph with respect to the outward unit normal, then
\[
\sup_{x \in B_{r/2}} \kappa_{\max}(x) \leq C
\] for some constant $C>0$ whose value depends on $n, k, r, \norm{u}_{L^\infty(B_r)}, \norm{Du}_{L^{\infty}(B_r)}$, $(\inf_{\cD} \psi)^{-1}$ and $\norm{\psi}_{C^{1,1}(\cD)}$, where
\[\cD:=\overline{B_{\frac{3r}{4}}} \times [-M, M], \quad M:=1+\norm{u}_{L^{\infty}(B_r)}.\]
\end{theorem}

Whether the theorem would hold for $k$-admissible solutions is unknown at the moment.
\end{remark}

\begin{remark} \label{concluding remark 2}
For the $\sigma_2$ equation and the $\sigma_n/\sigma_{n-2}$ equation, Guan-Qiu \cite{Guan-Qiu} and Liu \cite{Jianxiang} could get \eqref{claim} for $l_0=1$ and $l_0=n$, respectively. However, for the $\sigma_k/\sigma_{k-2}$ equation, we do not have one special index to work with. Instead, we have a range of indices and so getting \eqref{claim} is a subtle point. Our proof of \eqref{claim} is inspired by Guan-Qiu's trick \cite[(68)-(73)]{Guan-Qiu}.
\end{remark}

\section{Proof of Corollary \ref{our corollary}} \label{corollary proof}
\begin{proof}[Proof of Corollary \ref{our corollary}]
Let $n=3,4$ and let $r>0$. For a given positive function $\psi(x,u) \in C^{1,1}(\overline{B_r}\times \bR)$, suppose $u \in C^4(B_r)$ is a convex solution of the equation
\begin{equation}
\sum_{i=1}^{n} \arctan \left[\frac{\kappa_i(x)}{\psi(x,u(x))}\right]=\pi \quad \text{in $B_r$}, \label{the SL equation}
\end{equation} where $\kappa[u]=(\kappa_1,\ldots,\kappa_n)$ are the principal curvatures of its graph with respect to the outward unit normal.

For each $1 \leq i \leq n$, let
\[a_i:=\frac{\kappa_i}{\psi}, \quad \theta_i:=\arctan a_i.\] Since $\kappa_i \geq 0$ by the convexity of $u$ and $\psi$ is positive, we have
\[a_i \geq 0 \quad \text{and} \quad 0 \leq \theta_i<\frac{\pi}{2} \quad \text{for all $1 \leq i \leq n$ in $B_r$}.\] By the equation \eqref{the SL equation}, we have
\[\tan \left\{\sum_{i=1}^{n} \arctan \left[\frac{\kappa_i(x)}{\psi(x,u(x))}\right]\right\}=\tan \pi = 0.\] For $n=4$, by the angle sum identity for the tangent function, we have
\begin{equation}
0=\tan (\theta_1+\theta_2+\theta_3+\theta_4)=\frac{\sigma_1(a)-\sigma_3(a)}{1-\sigma_2(a)+\sigma_4(a)}. \label{the angle sum identity}
\end{equation} Since the numerator must vanish, it follows that
\[\frac{\sigma_1(\kappa)}{\psi}=\sigma_1(a)=\sigma_3(a)=\frac{\sigma_3(\kappa)}{\psi^3} \quad \text{in $B_r$}.\] Rearranging yields
\[\left(\frac{\sigma_3}{\sigma_1}\right)(\kappa_1(x),\ldots,\kappa_n(x))=[\psi(x,u(x))]^2 \quad \text{in $B_r$}.\] Note that, by the cosine angle sum identity, the denominator
\[1-\sigma_2(a)+\sigma_4(a)=\frac{\cos(\theta_1+\theta_2+\theta_3+\theta_4)}{\prod_{i=1}^{4} \cos \theta_i}=-\frac{1}{\underbrace{\prod_{i=1}^{4} \cos \theta_i}_{>0}}\] is indeed non-zero. This procedure works the same for $n=3$ because the numerator in \eqref{the angle sum identity} is still $\sigma_1(a)-\sigma_3(a)$ and the denominator is 
\[1-\sigma_2(a)=-\frac{1}{\prod_{i=1}^{3} \cos \theta_i}<0.\]

Now, since each $\theta_i<\pi/2$ and the sum of the angles is $\pi$, it follows that at least three angles are positive. Hence at least three principal curvatures are positive and so $\sigma_j(\kappa)>0$ for all $1 \leq j \leq 3$. This implies that the solution is $3$-admissible. For $n=3$, we apply Liu's result \cite{Jianxiang} on the quotient $\sigma_n/\sigma_{n-2}$ with the right-hand side $\psi^2$; and for $n=4$, we apply our Theorem \ref{our theorem 2} with the right-hand side $\psi^2$. In either case, we have
\[\sup_{B_{r/2}}\kappa_{\max} \leq C\] for some $C>0$ whose value depends on $n, r, \norm{u}_{L^\infty(B_r)}, \norm{Du}_{L^{\infty}(B_r)}$, $(\inf_{\cD} \psi^2)^{-1}$ and $\norm{\psi^2}_{C^{1,1}(\cD)}$, where
\[\cD:=\overline{B_{\frac{3r}{4}}} \times [-M, M], \quad M:=1+\norm{u}_{L^{\infty}(B_r)}.\] Since the dependence on $\psi^2$ is still controlled by $\inf_{\cD} \psi$ and $\norm{\psi}_{C^{1,1}(\cD)}$, we conclude that Corollary \ref{our corollary} follows.
\end{proof}

\section{Some related literature} \label{Literature}

In this section, we present a discussion of the literature and briefly review the existing methods along the way. In previous versions of this manuscript, this review was contained in the Introduction. However, since it would take up too many pages, we move it here in case the reader might still be interested.

We shall start with the $\sigma_k$ Hessian equations,
\[\sigma_k(\lambda(D^2u))=\psi(x,u(x)).\]

In dimension $n=2$, an interior Hessian bound for the Monge-\Ampere\ equation $\sigma_n(\lambda(D^2u))=\det(D^2u)=\lambda_1\lambda_2=1$ was obtained by Heinz \cite{Heinz} through the use of isothermal coordinates. In particular, this result was employed to solve the problem of constructing a closed convex surface which realizes a given metric on the sphere by a line element, according to the method of H.~Weyl. An elementary pointwise proof was given by Chen-Han-Ou \cite{Chen-Han-Ou}; see also \cite{Liu} for a proof using the partial Legendre transform.

When the dimension $n \geq 3$, Pogorelov \cite[Page 81]{Pogorelov-1978} showed that there is a generalized solution $u$ to the Monge-\Ampere\ equation $\det(D^2u)=\varphi(x)$ with a positive analytic right-hand side $\varphi$ but $u$ is not twice differentiable. In \cite{Urbas}, Urbas showed that this absence of regularity occurs to the $\sigma_k$ equations as well, in the sense that when $3 \leq k \leq n$, there exist viscosity solutions to the $\sigma_k$ equations with a positive smooth right-hand side but the solutions do not belong to $C^{1,\alpha}$ for any $\alpha>1-\frac{2}{k}$. For practical applications, the Pogorelov type interior $C^2$ estimates can often serve as an alternative; such estimates have been derived for the $\sigma_k$ equations in some cases; see \cite{Chou-Wang, Li-Ren-Wang, Zhang, Tu, Sheng-Urbas-Wang} and \cite[Lemma 17]{Ren-Wang-Xiao}.

Now, due to the examples of Pogorelov \cite{Pogorelov-1978} and Urbas \cite{Urbas}, it remains to study the regularity question for the $\sigma_2$ equation in dimensions $n \geq 3$. In \cite{WY}, Warren and Yuan were the first to confirm this regularity for $2$-admissible solutions to the equation $\sigma_2(\lambda(D^2u))=1$ in dimension $n=3$. Their proof is based on the observation that solutions of $\sigma_2(\lambda)=1$ in dimension 3 are also solutions to the special Lagrangian equation
\begin{equation}
\sum_{i=1}^{n} \arctan \lambda_i(D^2u)=\Theta \quad \text{with $n=3$ and $\Theta=\frac{\pi}{2}$}. \label{SL eqn}
\end{equation} Thus, the gradient graph $(x,Du(x))$ is Lagrangian and has zero mean curvature in $\bR^3 \times \bR^3$, according to Harvey and Lawson \cite[Theorem 2.3, Proposition 2.17]{HL}. Then, by the Michael-Simon mean value inequality \cite{Michael-Simon}, the subharmonic function $b=\log\sqrt{1+\lambda_{\max}^{2}(D^2u)}$ at any point is bounded by its integral around the point on the minimal Lagrangian graph. The task reduces to estimating the integral which is feasible because the function $b$ satisfies a so-called ``Jacobi inequality''.

Later, Qiu \cite{Qiu-1} extended Warren-Yuan's result to a non-constant right-hand side $\psi(x,u,Du)$. The solutions no longer solve a special Lagrangian equation but Qiu observed that the gradient graph $(x,Du(x))$ has bounded mean curvature as a submanifold of ($\bR^3 \times \bR^3$, $\psi dx^2+dy^2$) and some steps of Warren-Yuan can still be carried out. In order to overcome the additional complications that arise from the non-zero mean curvature, Qiu established a doubling inequality
\begin{equation}
\sup_{B_1} \Delta u \leq C(n,\norm{u}_{C^1(B_2)}) \sup_{B_{1/2}} \Delta u. \label{doubling}
\end{equation} Qiu also obtained interior curvature estimates for $2$-admissible solutions of the prescribed scalar curvature equation $\sigma_2(\kappa)=\psi(X,\nu)$ in dimension $n=3$; see \cite{Qiu-2}. In \cite{Zhou}, Zhou provided an alternative proof by recognizing the equation $\sigma_2=f^2$ as a ``twisted''  special Lagrangian equation
\begin{equation}
\sum_{i=1}^{n} \arctan \frac{\lambda_i(D^2u)}{f(x)}=\Theta \quad \text{with $n=3$ and $\Theta=\frac{\pi}{2}$}.\label{twisted}
\end{equation}

In \cite{SY-Annals}, Shankar and Yuan obtained the interior Hessian bound for $2$-admissible solutions to the equation $\sigma_2(\lambda(D^2u))=1$ in dimension $n=4$. For this dimension, there does not seem to be any Lagrangian structure available and the Jacobi inequality also fails. Instead, Shankar and Yuan derived an ``almost-Jacobi'' inequality which is weaker but applicable enough for them to get Qiu's doubling inequality \eqref{doubling}. The proof is different from the known integral methods; it employs an Alexandrov type twice differentiability theorem and a small perturbation theorem due to Savin \cite{Savin}, to bound the Hessian in a small inner ball by Qiu's doubling inequality \eqref{doubling}. Fan \cite{Fan} subsequently extended their result to a non-constant right-hand side by upgrading the almost-Jacobi inequality and the Alexandrov and the Savin parts to hold for a non-constant right-hand side.

For higher dimensions $n \geq 5$, the regularity question remains open for $2$-admissible solutions and we refer the reader to a discussion by Mooney \cite{Mooney-JMS}. On the other hand, the regularity in higher dimensions is possible if stronger convexity conditions are imposed. In \cite{Guan-Qiu}, Guan and Qiu established interior Hessian estimates for solutions satisfying $\sigma_3 \geq -A$ to the quadratic Hessian equation $\sigma_2(\lambda(D^2u))=\psi$ and the prescribed scalar curvature equation $\sigma_2(\kappa)=\psi$. This $\sigma_3$ lower bound constraint is particularly satisfied by convex solutions, thus generalizing the Heinz \cite{Heinz} interior estimates to higher dimensions for isometrically immersed hypersurfaces in $\bR^{n+1}$. A remarkable feature of Guan-Qiu's proof is that they applied the maximum principle to a novel choice of the auxiliary function and then proceeded with a series of delicate analysis. In other words, their proof is quite elementary. Mooney \cite{Mooney-PAMS} provided a short proof of this regularity result for convex solutions by estimating the dimension of the set of points where the solution agrees with a supporting hyperplane. Chen-Jian-Zhou \cite{Chen-Jian-Zhou} used an integral approach which is more complicated but they could weaken the regularity assumption on the right-hand side $\psi(x)$ from $C^{1,1}$ to $C^{0,1}$.

Finally, for non-convex solutions in dimensions $n \geq 5$, if $u$ is only slightly non-convex in the sense that
\[D^2u \geq \left[\delta-\sqrt{\frac{2}{n(n-1)}}\ \right]\ I,\] where $\delta>0$ is arbitrary and $I$ is the $n \times n$ identity matrix, then the Legendre-Lewy transform $w$ of $u$ satisfies a uniformly elliptic, concave equation with bounded Hessian. McGonagle, Song, and Yuan showed that the interior Hessian estimate would follow from a contradiction argument combined with the classical Evans-Krylov theorem \cite{Evans, Krylov} and the constant rank theorem of Caffarelli-Guan-Ma \cite{CGM}; see \cite{MSY} for the proof. The observation that $w$ satisfies a concave equation under this particular lower bound was also applied to obtain rigidity of entire solutions and interior regularity of viscosity solutions; see Chang-Yuan \cite{CY} and Shankar-Yuan \cite{SY-JMS}, respectively. For a general semi-convex solution $u$, i.e., $D^2u \geq -KI$ for an arbitrary $K>0$, the level set of the Legendre-Lewy transform $w$ would no longer be convex for a large $K$. To compensate for this saddle-ness, Shankar and Yuan \cite{SY-CVPDE} made a change of variables by the Legendre-Lewy transformation and showed that the quantity $b=\log \lambda_{\max}(D^2u)$ is a subsolution of a uniformly elliptic equation. This transformation enabled them to deduce a mean value inequality
\[b(0) \leq C(n,K) \int_{B_1} b(x) \Delta u(x)\ dx\] by the local maximum principle. The Hessian bound then follows by estimating the integral, which is again a feasible task as in the three dimensional case because $b$ satisfies the Jacobi inequality for semi-convex solutions in all dimensions and the linearized operator of $\sigma_2$ is in divergence form. The Hessian bound can also be obtained under a weaker, dynamic semi-convexity condition by Shankar-Yuan's new doubling proof; see their paper \cite{SY-Annals}. It remains to study a priori Hessian estimates and the interior regularity for $2$-admissible solutions of the $\sigma_2$ Hessian equation and the $\sigma_2$ curvature equation in dimension $n \geq 5$.

We now turn to the a priori Hessian estimates for quotient equations. In \cite{Lu-2}, Lu studied the following Hessian quotient equations
\begin{equation}
\left(\frac{\sigma_k}{\sigma_{l}}\right)(\lambda_1(x),\ldots,\lambda_n(x))=\psi(x,u(x)), \quad x \in B_{10}, \label{Hessian quotient equation}
\end{equation} where $\lambda_1(x),\ldots,\lambda_n(x)$ are the eigenvalues of the Hessian matrix $D^2u$ of a smooth function $u$. For $n$-admissible solutions to either of the following two quotients,
\begin{align*}
&\ \left(\frac{\sigma_n}{\sigma_{n-1}}\right)(\lambda_1(x),\ldots,\lambda_n(x))=\psi(x,u(x)) \\
\text{or}\quad &\ \left(\frac{\sigma_n}{\sigma_{n-2}}\right)(\lambda_1(x),\ldots,\lambda_n(x))=\psi(x,u(x)),
\end{align*} Lu was able to obtain purely interior estimates,
\[|D^2u(0)| \leq C,\] for some $C>0$ depending on known data via an integral approach; see also the recent preprints \cite{Fung} and \cite{Jianxiang-2} for alternative proofs using the doubling inequality and the compactness argument, respectively. Note that, according to Guan-Ren-Wang \cite[Theorem 1.2]{Guan-Ren-Wang} and Fung \cite[Remark 1.3]{Fung}, $C^2$ estimates may fail for the quotient equation if the right-hand side contains gradient terms. In \cite{Lu-2}, Lu also showed that the Hessian quotient equations \eqref{Hessian quotient equation} with $\psi=\psi(x)$ admit non-classical solutions when $k-l \geq 3$. Therefore, Lu's work \cite{Lu-2} settles, for a large part, the interior regularity problem for Hessian quotient equations. It thus remains to study whether purely local Hessian estimates are possible for $k$-admissible solutions to the following two quotients,
\begin{equation}
\left(\frac{\sigma_k}{\sigma_{k-1}}\right)(\lambda(D^2u))=\psi(x,u) \quad \text{and} \quad \left(\frac{\sigma_k}{\sigma_{k-2}}\right)(\lambda(D^2u))=\psi(x,u), \label{the two quotients}
\end{equation} when $k<n$. 

In a subsequent work \cite{LT-Crelle}, Lu and Tsai obtained Pogorelov type interior Hessian estimates for the Hessian quotient equations $(\sigma_n/\sigma_k)(\lambda(D^2u))=\psi$. In another subsequent work \cite{LT-CPAA}, Lu and Tsai showed that the proof in \cite{Lu-2} can be applied verbatim to obtain purely interior Hessian estimates for convex solutions to the two quotients in \eqref{the two quotients} if one can derive some concavity inequalities for the two operators. The required inequalities were first available for the quotients $\sigma_n/\sigma_k$ by Guan-Sroka \cite{Guan-Sroka}, and they comprise one of the major ingredients for the proofs in \cite{Lu-2} and \cite{LT-Crelle}.

Concerning the two quotients in \eqref{the two quotients}, local estimates for $3$-admissible solutions to the quotient equation 
\begin{equation}
\left(\frac{\sigma_3}{\sigma_1}\right)(\lambda(D^2u))=1 \quad \text{in dimension $n=4$} \label{quotient in dimension 4}
\end{equation} were known because it is the special Lagrangian equation \eqref{SL eqn} with $n=4$ and $\Theta=\pi$, for which the local estimates were obtained by Wang and Yuan \cite{Wang-Yuan}. In \cite{Zhou}, Zhou extended the local estimates for \eqref{quotient in dimension 4} to a non-constant right-hand side $f^2(x)$ by recognizing it as a ``twisted'' special Lagrangian equation \eqref{twisted} with $n=4$ and $\Theta=\pi$. In a recent preprint \cite{Lu-Sroka}, Lu and Sroka made a keen observation that if $u$ is a $2$-admissible solution to
\[\left(\frac{\sigma_2}{\sigma_1}\right)(\lambda(D^2u))=1,\] then the function
\[v:=u-\frac{1}{2(n-1)}|x|^2\] is a $2$-admissible solution to
\[\sigma_2(\lambda(D^2v))=\frac{n}{2(n-1)}.\] Since interior regularity is available for the $\sigma_2$ Hessian equation in several cases as we have reviewed, Lu and Sroka can then conclude directly the same results for the $\sigma_2/\sigma_1$ Hessian quotient equation in all those known cases, by converting information about $v$ back into $u$. In another preprint \cite{Jiao-Sui}, Jiao and Sui obtained the interior estimates for $2$-admissible solutions to $\sigma_2/\sigma_1$ with a variable right-hand side $\psi=\psi(x,u)$ in dimension $n=3$ and for semi-convex solutions in all dimensions, by following the doubling proof of Shankar-Yuan \cite{SY-Annals}. The key is that they were able to derive the doubling inequality \eqref{doubling} for the equation $\sigma_2/\sigma_1=\psi$. Whether one can conclude the result for strictly $2$-convex solutions in dimension $n=4$ with a non-constant right-hand side is unclear at the time when we wrote this manuscript.

In \cite{Mei-Yan}, Mei and Yan followed the observation of Lu-Sroka \cite{Lu-Sroka} to find that if $u$ is a semi-convex solutions to 
\[\left(\frac{\sigma_3}{\sigma_l}\right)(\lambda(D^2u))=1, \quad \text{where $l=1,2$},\] then 
\[v(x):=u(x)-\frac{a}{2}|x|^2\] is a semi-convex solution to
\[\sigma_3(\lambda(D^2v))+C(n,a) \cdot \sigma_2(\lambda(D^2v))=1,\] by choosing appropriate values for the constant $a$. Proceeding in the way of, e.g., Lu \cite{Lu-2} or Shankar-Yuan \cite{SY-CVPDE}, they were able to derive interior Hessian estimates for semi-convex solutions $v$ to the above ``sum Hessian'' equations in all dimensions because a concavity inequality is available for this sum Hessian operator, due to the work of Dong-Xu-Zhang \cite{Dong-Xu-Zhang}. Converting back into $u$, one obtain the desired interior estimates for semi-convex $u$ solving the Hessian quotients $\sigma_3/\sigma_l=1$ in all dimensions. Whether these estimates are possible for $3$-admissible solutions remains unclear to us at the moment. 

Most recently, for convex solutions, the concavity inequality for $\sigma_k/\sigma_{k-1}$ has been independently established by Tsai \cite{Tsai} and Li-Wu \cite{Li-Wu}. In addition, Li and Wu \cite{Li-Wu} also established the concavity inequality for convex solutions of the quotient $\sigma_k/\sigma_{k-2}$. Therefore, the desired interior Hessian estimates for convex solutions to the quotients in \eqref{the two quotients} readily follow as a consequence of Lu-Tsai's framework in \cite{LT-CPAA}. In \cite{Dong-Zhang}, Dong and Zhang improved the concavity inequalities for both operators so that they now hold for semi-convex solutions, i.e., $D^2u \geq -KI$ for some constant $K \geq 0$. Hence, the interior Hessian estimates can be established for semi-convex solutions of these two quotients. Whether one can obtain he estimates for $k$-admissible solutions remains still open.

For the prescribed curvature quotient equations,
\[\left(\frac{\sigma_k}{\sigma_l}\right)(\kappa[u])=\psi(x,u), \quad 1 \leq l < k \leq n,\] the Pogorelov type interior curvature estimates for $k$-admissible solutions follow from a more general result by Sheng, Urbas, and Wang in their remarkable paper \cite{Sheng-Urbas-Wang}. An elementary proof of the purely interior curvature estimates for $k$-admissible solutions to the particular quotient $(\sigma_k/\sigma_{k-1})(\kappa)$, also follow from their argument; see \cite[Remark 2.1]{Sheng-Urbas-Wang}. This is due to the special algebraic structure of the operator $\sigma_k/\sigma_{k-1}$ and the presence of an extra positive curvature term arising from the interchanging formula for covariant differentiations of the second fundamental form. However, the other curvature quotient $(\sigma_k/\sigma_{k-2})(\kappa)$ no longer possesses the same structural advantages and hence the same proof cannot be applied in this case. Recently, in a preprint \cite{Jianxiang}, Liu proved the interior curvature estimates for this curvature quotient when $k=n$, and the proof is also of an elementary pointwise nature. It relies on a concavity inequality for the $\sigma_n/\sigma_{n-2}$ operator which was derived by Guan-Sroka \cite{Guan-Sroka} and Lu \cite{Lu-2}, and a minor modification of Guan-Qiu's delicate analysis in their pointwise proof \cite{Guan-Qiu} for the $\sigma_2$ equations. In Theorem \ref{our theorem 2}, we prove the interior curvature estimates for convex solutions of the prescribed curvature quotient equation
\[\left(\frac{\sigma_k}{\sigma_{k-2}}\right)(\kappa[u])=\psi(x,u(x)) \quad \text{when $k<n$}.\] Our proof makes use of the concavity inequality derived by Li-Wu \cite{Li-Wu} and Dong-Zhang \cite{Dong-Zhang} for the quotient operator $\sigma_k/\sigma_{k-2}$ and still applies the maximum principle to the novel auxiliary function of Guan-Qiu \cite{Guan-Qiu}. The proof is again elementary and the new contribution is that we are able to handle the difficulties that arise from the substantially more complicated structure of the quotient equation when $3 \leq k \leq n-1$; see Remark \ref{concluding remark 1}. Note that our proof applies to $(k+1)$-admissible solutions as well; see Remark \ref{concluding remark 1}.

\section*{Acknowledgments}
This work was begun at The Chinese University of Hong Kong where the author was a Ph.D. candidate. The author would like to thank Professor Man-Chun Lee and Professor Xiaolu Tan for their support. The author is also grateful to Jianxiang Liu for pointing out a gap in a preliminary version of this manuscript, and to Professor Guohuan Qiu for some kind comments. Research of the author was supported by the China Postdoctoral Science Foundation under Grant Number 2026M793406.

\bibliography{refs}

\begin{bibdiv}
\begin{biblist}

\bib{BCD}{article}{
      author={Brendle, Simon},
      author={Choi, Kyeongsu},
      author={Daskalopoulos, Panagiota},
       title={Asymptotic behavior of flows by powers of the {G}aussian
  curvature},
        date={2017},
        ISSN={0001-5962,1871-2509},
     journal={Acta Math.},
      volume={219},
      number={1},
       pages={1\ndash 16},
         url={https://doi.org/10.4310/ACTA.2017.v219.n1.a1},
      review={\MR{3765656}},
}

\bib{BV}{article}{
      author={Brendle, Simon},
      author={Viaclovsky, Jeff~A.},
       title={A variational characterization for {$\sigma_{n/2}$}},
        date={2004},
        ISSN={0944-2669,1432-0835},
     journal={Calc. Var. Partial Differential Equations},
      volume={20},
      number={4},
       pages={399\ndash 402},
         url={https://doi.org/10.1007/s00526-003-0234-9},
      review={\MR{2071927}},
}

\bib{CNS3}{article}{
      author={Caffarelli, L.},
      author={Nirenberg, L.},
      author={Spruck, J.},
       title={The {D}irichlet problem for nonlinear second-order elliptic
  equations. {III}. {F}unctions of the eigenvalues of the {H}essian},
        date={1985},
        ISSN={0001-5962,1871-2509},
     journal={Acta Math.},
      volume={155},
      number={3-4},
       pages={261\ndash 301},
         url={https://doi.org/10.1007/BF02392544},
      review={\MR{806416}},
}

\bib{CNS4}{incollection}{
      author={Caffarelli, L.},
      author={Nirenberg, L.},
      author={Spruck, J.},
       title={Nonlinear second order elliptic equations. {IV}. {S}tarshaped
  compact {W}eingarten hypersurfaces},
        date={1986},
   booktitle={Current topics in partial differential equations},
   publisher={Kinokuniya, Tokyo},
       pages={1\ndash 26},
      review={\MR{1112140}},
}

\bib{CGM}{article}{
      author={Caffarelli, Luis},
      author={Guan, Pengfei},
      author={Ma, Xi-Nan},
       title={A constant rank theorem for solutions of fully nonlinear elliptic
  equations},
        date={2007},
        ISSN={0010-3640,1097-0312},
     journal={Comm. Pure Appl. Math.},
      volume={60},
      number={12},
       pages={1769\ndash 1791},
         url={https://doi.org/10.1002/cpa.20197},
      review={\MR{2358648}},
}

\bib{CNS5}{article}{
      author={Caffarelli, Luis},
      author={Nirenberg, Louis},
      author={Spruck, Joel},
       title={Nonlinear second-order elliptic equations. {V}. {T}he {D}irichlet
  problem for {W}eingarten hypersurfaces},
        date={1988},
        ISSN={0010-3640,1097-0312},
     journal={Comm. Pure Appl. Math.},
      volume={41},
      number={1},
       pages={47\ndash 70},
         url={https://doi.org/10.1002/cpa.3160410105},
      review={\MR{917124}},
}

\bib{CY}{article}{
      author={Chang, Sun-Yung~Alice},
      author={Yuan, Yu},
       title={A {L}iouville problem for the sigma-2 equation},
        date={2010},
        ISSN={1078-0947,1553-5231},
     journal={Discrete Contin. Dyn. Syst.},
      volume={28},
      number={2},
       pages={659\ndash 664},
         url={https://doi.org/10.3934/dcds.2010.28.659},
      review={\MR{2644763}},
}

\bib{Chen-Han-Ou}{article}{
      author={Chen, Chuanqiang},
      author={Han, Fei},
      author={Ou, Qianzhong},
       title={The interior {$C^2$} estimate for the {M}onge-{A}mp\`ere equation
  in dimension {$n=2$}},
        date={2016},
        ISSN={2157-5045,1948-206X},
     journal={Anal. PDE},
      volume={9},
      number={6},
       pages={1419\ndash 1432},
         url={https://doi.org/10.2140/apde.2016.9.1419},
      review={\MR{3555315}},
}

\bib{Chen-Jian-Zhou}{article}{
      author={Chen, Ruosi},
      author={Jian, Huaiyu},
      author={Zhou, Xingchen},
       title={An integral approach to prescribing scalar curvature equations},
        date={2026},
        ISSN={0025-5831,1432-1807},
     journal={Math. Ann.},
      volume={394},
      number={3},
       pages={Paper No. 72, 29},
         url={https://doi.org/10.1007/s00208-026-03312-z},
      review={\MR{5035624}},
}

\bib{Chou-Wang}{article}{
      author={Chou, Kai-Seng},
      author={Wang, Xu-Jia},
       title={A variational theory of the {H}essian equation},
        date={2001},
        ISSN={0010-3640,1097-0312},
     journal={Comm. Pure Appl. Math.},
      volume={54},
      number={9},
       pages={1029\ndash 1064},
         url={https://doi.org/10.1002/cpa.1016},
      review={\MR{1835381}},
}

\bib{Dong-Xu-Zhang}{misc}{
      author={Dong, Weisong},
      author={Xu, Sirui},
      author={Zhang, Ruijia},
       title={Pogorelov interior estimates for sum-of-{H}essians equations},
        date={preprint on arXiv},
  note={\href{https://arxiv.org/abs/2603.15345}{https://arxiv.org/abs/2603.15345}},
}

\bib{Dong-Zhang}{misc}{
      author={Dong, Weisong},
      author={Zhang, Ruijia},
       title={Interior {H}essian estimates for {H}essian quotient equations},
        date={preprint on arXiv},
  note={\href{https://arxiv.org/abs/2608.19087}{https://arxiv.org/abs/2608.19087}},
}

\bib{Evans}{article}{
      author={Evans, Lawrence~C.},
       title={Classical solutions of fully nonlinear, convex, second-order
  elliptic equations},
        date={1982},
        ISSN={0010-3640,1097-0312},
     journal={Comm. Pure Appl. Math.},
      volume={35},
      number={3},
       pages={333\ndash 363},
         url={https://doi.org/10.1002/cpa.3160350303},
      review={\MR{649348}},
}

\bib{Fan}{article}{
      author={Fan, Zhenyu},
       title={Hessian estimates for the sigma-2 equation with variable
  right-hand side terms in dimension 4},
        date={2026},
        ISSN={0001-8708,1090-2082},
     journal={Adv. Math.},
      volume={494},
       pages={Paper No. 110953},
         url={https://doi.org/10.1016/j.aim.2026.110953},
      review={\MR{5059337}},
}

\bib{Fung}{misc}{
      author={Fung, Cheuk~Yan},
       title={Doubling argument of the {H}essian estimate for the {H}essian
  quotient equations},
        date={preprint on arXiv},
  note={\href{https://arxiv.org/abs/2607.21982}{https://arxiv.org/abs/2607.21982}},
}

\bib{Guan-Guan}{article}{
      author={Guan, Bo},
      author={Guan, Pengfei},
       title={Convex hypersurfaces of prescribed curvatures},
        date={2002},
        ISSN={0003-486X,1939-8980},
     journal={Ann. of Math. (2)},
      volume={156},
      number={2},
       pages={655\ndash 673},
         url={https://doi.org/10.2307/3597202},
      review={\MR{1933079}},
}

\bib{Guan-Li}{article}{
      author={Guan, Pengfei},
      author={Li, Junfang},
       title={The quermassintegral inequalities for {$k$}-convex starshaped
  domains},
        date={2009},
        ISSN={0001-8708,1090-2082},
     journal={Adv. Math.},
      volume={221},
      number={5},
       pages={1725\ndash 1732},
         url={https://doi.org/10.1016/j.aim.2009.03.005},
      review={\MR{2522433}},
}

\bib{Guan-Li-Li}{article}{
      author={Guan, Pengfei},
      author={Li, Junfang},
      author={Li, Yanyan},
       title={Hypersurfaces of prescribed curvature measure},
        date={2012},
        ISSN={0012-7094,1547-7398},
     journal={Duke Math. J.},
      volume={161},
      number={10},
       pages={1927\ndash 1942},
         url={https://doi.org/10.1215/00127094-1645550},
      review={\MR{2954620}},
}

\bib{Guan-Ma}{article}{
      author={Guan, Pengfei},
      author={Ma, Xi-Nan},
       title={The {C}hristoffel-{M}inkowski problem. {I}. {C}onvexity of
  solutions of a {H}essian equation},
        date={2003},
        ISSN={0020-9910,1432-1297},
     journal={Invent. Math.},
      volume={151},
      number={3},
       pages={553\ndash 577},
         url={https://doi.org/10.1007/s00222-002-0259-2},
      review={\MR{1961338}},
}

\bib{Guan-Qiu}{article}{
      author={Guan, Pengfei},
      author={Qiu, Guohuan},
       title={Interior {$C^2$} regularity of convex solutions to prescribing
  scalar curvature equations},
        date={2019},
        ISSN={0012-7094,1547-7398},
     journal={Duke Math. J.},
      volume={168},
      number={9},
       pages={1641\ndash 1663},
         url={https://doi.org/10.1215/00127094-2019-0001},
      review={\MR{3961212}},
}

\bib{Guan-Ren-Wang}{article}{
      author={Guan, Pengfei},
      author={Ren, Changyu},
      author={Wang, Zhizhang},
       title={Global {$C^2$}-estimates for convex solutions of curvature
  equations},
        date={2015},
        ISSN={0010-3640,1097-0312},
     journal={Comm. Pure Appl. Math.},
      volume={68},
      number={8},
       pages={1287\ndash 1325},
         url={https://doi.org/10.1002/cpa.21528},
      review={\MR{3366747}},
}

\bib{Guan-Sroka}{article}{
      author={Guan, Pengfei},
      author={Sroka, Marcin},
       title={A special concavity property for positive {H}essian quotient
  operators},
        date={2026},
        ISSN={1078-0947,1553-5231},
     journal={Discrete Contin. Dyn. Syst.},
      volume={54},
       pages={50\ndash 60},
         url={https://doi.org/10.3934/dcds.2025181},
      review={\MR{5097161}},
}

\bib{GVW}{article}{
      author={Guan, Pengfei},
      author={Viaclovsky, Jeff},
      author={Wang, Guofang},
       title={Some properties of the {S}chouten tensor and applications to
  conformal geometry},
        date={2003},
        ISSN={0002-9947,1088-6850},
     journal={Trans. Amer. Math. Soc.},
      volume={355},
      number={3},
       pages={925\ndash 933},
         url={https://doi.org/10.1090/S0002-9947-02-03132-X},
      review={\MR{1938739}},
}

\bib{HL}{article}{
      author={Harvey, Reese},
      author={Lawson, H.~Blaine, Jr.},
       title={Calibrated geometries},
        date={1982},
        ISSN={0001-5962,1871-2509},
     journal={Acta Math.},
      volume={148},
       pages={47\ndash 157},
         url={https://doi.org/10.1007/BF02392726},
      review={\MR{666108}},
}

\bib{Heinz}{article}{
      author={Heinz, Erhard},
       title={On elliptic {M}onge-{A}mp\`ere equations and {W}eyl's embedding
  problem},
        date={1959},
        ISSN={0021-7670,1565-8538},
     journal={J. Analyse Math.},
      volume={7},
       pages={1\ndash 52},
         url={https://doi.org/10.1007/BF02787679},
      review={\MR{111943}},
}

\bib{Ivochkina-1}{article}{
      author={Ivochkina, N.~M.},
       title={Solution of the {D}irichlet problem for certain equations of
  {M}onge-{A}mp\`ere type},
        date={1985},
        ISSN={0368-8666},
     journal={Mat. Sb. (N.S.)},
      volume={128(170)},
      number={3},
       pages={403\ndash 415, 447},
      review={\MR{815272}},
}

\bib{Ivochkina-2}{article}{
      author={Ivochkina, N.~M.},
       title={Solution of the {D}irichlet problem for equations of {$m$}th
  order curvature},
        date={1989},
        ISSN={0368-8666},
     journal={Mat. Sb.},
      volume={180},
      number={7},
       pages={867\ndash 887, 991},
         url={https://doi.org/10.1070/SM1990v067n02ABEH002089},
      review={\MR{1014618}},
}

\bib{Jiao-Sui}{misc}{
      author={Jiao, Heming},
      author={Sui, Zhenan},
       title={Interior {H}essian estimates for {H}essian quotient equations in
  dimension three},
        date={preprint in arXiv},
  note={\href{https://arxiv.org/abs/2602.14064}{https://arxiv.org/abs/2602.14064}},
}

\bib{Krylov}{article}{
      author={Krylov, N.~V.},
       title={Boundedly inhomogeneous elliptic and parabolic equations},
        date={1982},
        ISSN={0373-2436},
     journal={Izv. Akad. Nauk SSSR Ser. Mat.},
      volume={46},
      number={3},
       pages={487\ndash 523, 670},
      review={\MR{661144}},
}

\bib{Krylov-1987}{article}{
      author={Krylov, N.~V.},
       title={On the first boundary value problem for nonlinear degenerate
  elliptic equations},
        date={1987},
        ISSN={0373-2436},
     journal={Izv. Akad. Nauk SSSR Ser. Mat.},
      volume={51},
      number={2},
       pages={242\ndash 269, 446},
         url={https://doi.org/10.1070/IM1988v030n02ABEH001002},
      review={\MR{896996}},
}

\bib{Li-Ren-Wang}{article}{
      author={Li, Ming},
      author={Ren, Changyu},
      author={Wang, Zhizhang},
       title={An interior estimate for convex solutions and a rigidity
  theorem},
        date={2016},
        ISSN={0022-1236,1096-0783},
     journal={J. Funct. Anal.},
      volume={270},
      number={7},
       pages={2691\ndash 2714},
         url={https://doi.org/10.1016/j.jfa.2016.01.008},
      review={\MR{3464054}},
}

\bib{Li-Wu}{misc}{
      author={Li, Zhisu},
      author={Wu, Ke},
       title={A concavity inequality and interior {$C^2$} estimate for
  {H}essian quotient equations},
        date={preprint on arXiv},
  note={\href{https://arxiv.org/abs/2608.17405}{https://arxiv.org/abs/2608.17405}},
}

\bib{Lin-Trudinger}{article}{
      author={Lin, Mi},
      author={Trudinger, Neil~S.},
       title={On some inequalities for elementary symmetric functions},
        date={1994},
        ISSN={0004-9727},
     journal={Bull. Austral. Math. Soc.},
      volume={50},
      number={2},
       pages={317\ndash 326},
         url={https://doi.org/10.1017/S0004972700013770},
      review={\MR{1296759}},
}

\bib{Liu}{article}{
      author={Liu, Jiakun},
       title={Interior {$C^2$} estimate for {M}onge-{A}mp\`ere equation in
  dimension two},
        date={2021},
        ISSN={0002-9939,1088-6826},
     journal={Proc. Amer. Math. Soc.},
      volume={149},
      number={6},
       pages={2479\ndash 2486},
         url={https://doi.org/10.1090/proc/15459},
      review={\MR{4246799}},
}

\bib{Jianxiang-2}{misc}{
      author={Liu, Jianxiang},
       title={Hessian estimates for convex solutions of {H}essian quotient
  equations via compactness},
        date={preprint on arXiv},
  note={\href{https://arxiv.org/abs/2610.08442}{https://arxiv.org/abs/2610.08442}},
}

\bib{Jianxiang}{misc}{
      author={Liu, Jianxiang},
       title={Interior curvature estimate for curvature quotient equations on
  convex hypersurfaces},
        date={preprint on arXiv},
  note={\href{https://arxiv.org/abs/2505.00360}{https://arxiv.org/abs/2505.00360}},
}

\bib{Lu-2}{article}{
      author={Lu, Siyuan},
       title={Interior {$C^2$} estimate for {H}essian quotient equation in
  general dimension},
        date={2025},
        ISSN={2524-5317,2199-2576},
     journal={Ann. PDE},
      volume={11},
      number={2},
       pages={Paper No. 17, 26},
         url={https://doi.org/10.1007/s40818-025-00215-1},
      review={\MR{4926480}},
}

\bib{Lu-Sroka}{article}{
      author={Lu, Siyuan},
      author={Sroka, Marcin},
       title={On {L}iouville's theorem for the {H}essian quotient equation
  {$\sigma_2/\sigma_1$}},
        date={to appear},
     journal={Int. Math. Res. Not. IMRN},
  note={\href{https://arxiv.org/abs/2602.14946}{https://arxiv.org/abs/2602.14946}},
}

\bib{LT-CPAA}{article}{
      author={Lu, Siyuan},
      author={Tsai, Yi-Lin},
       title={A note on interior {$C^2$} estimate for general {H}essian
  quotient equation},
        date={2026},
        ISSN={1534-0392,1553-5258},
     journal={Commun. Pure Appl. Anal.},
      volume={33},
       pages={88\ndash 100},
         url={https://doi.org/10.3934/cpaa.2026045},
      review={\MR{5097018}},
}

\bib{LT-Crelle}{article}{
      author={Lu, Siyuan},
      author={Tsai, Yi-Lin},
       title={Pogorelov type interior {$C^{2}$} estimate for {H}essian quotient
  equation and its application},
        date={2026},
        ISSN={0075-4102,1435-5345},
     journal={J. Reine Angew. Math.},
      volume={831},
       pages={155\ndash 184},
         url={https://doi.org/10.1515/crelle-2025-0091},
      review={\MR{5025057}},
}

\bib{MSY}{article}{
      author={McGonagle, Matt},
      author={Song, Chong},
      author={Yuan, Yu},
       title={Hessian estimates for convex solutions to quadratic {H}essian
  equation},
        date={2019},
        ISSN={0294-1449,1873-1430},
     journal={Ann. Inst. H. Poincar\'e{} C Anal. Non Lin\'eaire},
      volume={36},
      number={2},
       pages={451\ndash 454},
         url={https://doi.org/10.1016/j.anihpc.2018.07.001},
      review={\MR{3913193}},
}

\bib{Mei-Yan}{misc}{
      author={Mei, Xinqun},
      author={Yan, Jin},
       title={Interior {$C^{2}$} estimate for semi-convex solutions to a class
  of {H}essian quotient equations in arbitrary dimensions},
        date={preprint on arXiv},
  note={\href{https://arxiv.org/abs/2604.23349}{https://arxiv.org/abs/2604.23349}},
}

\bib{Michael-Simon}{article}{
      author={Michael, J.~H.},
      author={Simon, L.~M.},
       title={Sobolev and mean-value inequalities on generalized submanifolds
  of {$R\sp{n}$}},
        date={1973},
        ISSN={0010-3640,1097-0312},
     journal={Comm. Pure Appl. Math.},
      volume={26},
       pages={361\ndash 379},
         url={https://doi.org/10.1002/cpa.3160260305},
      review={\MR{344978}},
}

\bib{Mooney-PAMS}{article}{
      author={Mooney, Connor},
       title={Strict 2-convexity of convex solutions to the quadratic {H}essian
  equation},
        date={2021},
        ISSN={0002-9939,1088-6826},
     journal={Proc. Amer. Math. Soc.},
      volume={149},
      number={6},
       pages={2473\ndash 2477},
         url={https://doi.org/10.1090/proc/15454},
      review={\MR{4246798}},
}

\bib{Mooney-JMS}{article}{
      author={Mooney, Connor},
       title={Remarks on the quadratic {H}essian equation},
        date={to appear},
     journal={J. Math. Study},
       pages={preprint available on arXiv:
  \href{https://arxiv.org/abs/2505.14586}{https://arxiv.org/abs/2505.14586}},
}

\bib{Pogorelov-1978}{book}{
      author={Pogorelov, Aleksey Vasil\cprime~yevich},
       title={The {M}inkowski multidimensional problem},
      series={Scripta Series in Mathematics},
   publisher={V. H. Winston \& Sons, Washington, DC; Halsted Press [John Wiley
  \& Sons], New York-Toronto-London},
        date={1978},
        ISBN={0-470-99358-8},
        note={Translated from the Russian by Vladimir Oliker, Introduction by
  Louis Nirenberg},
      review={\MR{478079}},
}

\bib{Qiu-2}{article}{
      author={Qiu, Guohuan},
       title={Interior curvature estimates for hypersurfaces of prescribing
  scalar curvature in dimension three},
        date={2024},
        ISSN={0002-9327,1080-6377},
     journal={Amer. J. Math.},
      volume={146},
      number={3},
       pages={579\ndash 605},
         url={https://doi.org/10.1353/ajm.2024.a928319},
      review={\MR{4752675}},
}

\bib{Qiu-1}{article}{
      author={Qiu, Guohuan},
       title={Interior {H}essian estimates for {$\sigma_2$} equations in
  dimension three},
        date={2024},
        ISSN={2731-8648,2731-8656},
     journal={Front. Math.},
      volume={19},
      number={4},
       pages={577\ndash 598},
         url={https://doi.org/10.1007/s11464-023-0156-0},
      review={\MR{4768566}},
}

\bib{Qiu-Zhou}{misc}{
      author={Qiu, Guohuan},
      author={Zhou, Xingchen},
       title={A priori interior estimates for special {L}agrangian curvature
  equations},
        date={preprint on arXiv},
  note={\href{https://arxiv.org/abs/2407.15159}{https://arxiv.org/abs/2407.15159}},
}

\bib{Ren-Wang-2}{article}{
      author={Ren, Changyu},
      author={Wang, Zhizhang},
       title={The global curvature estimate for the {$n-2$} {H}essian
  equation},
        date={2023},
        ISSN={0944-2669,1432-0835},
     journal={Calc. Var. Partial Differential Equations},
      volume={62},
      number={9},
       pages={Paper No. 239, 50},
         url={https://doi.org/10.1007/s00526-023-02570-y},
      review={\MR{4646879}},
}

\bib{Ren-Wang-Xiao}{article}{
      author={Ren, Changyu},
      author={Wang, Zhizhang},
      author={Xiao, Ling},
       title={The prescribed curvature problem for entire hypersurfaces in
  {M}inkowski space},
        date={2024},
        ISSN={2157-5045,1948-206X},
     journal={Anal. PDE},
      volume={17},
      number={1},
       pages={1\ndash 40},
         url={https://doi.org/10.2140/apde.2024.17.1},
      review={\MR{4702314}},
}

\bib{Savin}{article}{
      author={Savin, Ovidiu},
       title={Small perturbation solutions for elliptic equations},
        date={2007},
        ISSN={0360-5302,1532-4133},
     journal={Comm. Partial Differential Equations},
      volume={32},
      number={4-6},
       pages={557\ndash 578},
         url={https://doi.org/10.1080/03605300500394405},
      review={\MR{2334822}},
}

\bib{SY-CVPDE}{article}{
      author={Shankar, Ravi},
      author={Yuan, Yu},
       title={Hessian estimate for semiconvex solutions to the sigma-2
  equation},
        date={2020},
        ISSN={0944-2669,1432-0835},
     journal={Calc. Var. Partial Differential Equations},
      volume={59},
      number={1},
       pages={Paper No. 30, 12},
         url={https://doi.org/10.1007/s00526-019-1690-1},
      review={\MR{4054864}},
}

\bib{SY-JMS}{article}{
      author={Shankar, Ravi},
      author={Yuan, Yu},
       title={Regularity for almost convex viscosity solutions of the sigma-2
  equation},
        date={2021},
        ISSN={2096-9856,2617-8702},
     journal={J. Math. Study},
      volume={54},
      number={2},
       pages={164\ndash 170},
         url={https://doi.org/10.4208/jms.v54n2.21.03},
      review={\MR{4210289}},
}

\bib{SY-Annals}{article}{
      author={Shankar, Ravi},
      author={Yuan, Yu},
       title={Hessian estimates for the sigma-2 equation in dimension four},
        date={2025},
        ISSN={0003-486X,1939-8980},
     journal={Ann. of Math. (2)},
      volume={201},
      number={2},
       pages={489\ndash 513},
         url={https://doi.org/10.4007/annals.2025.201.2.4},
      review={\MR{4880431}},
}

\bib{Sheng-Urbas-Wang}{article}{
      author={Sheng, Weimin},
      author={Urbas, John},
      author={Wang, Xu-Jia},
       title={Interior curvature bounds for a class of curvature equations},
        date={2004},
        ISSN={0012-7094,1547-7398},
     journal={Duke Math. J.},
      volume={123},
      number={2},
       pages={235\ndash 264},
         url={https://doi.org/10.1215/S0012-7094-04-12321-8},
      review={\MR{2066938}},
}

\bib{Smith}{article}{
      author={Smith, Graham},
       title={Special {L}agrangian curvature},
        date={2013},
        ISSN={0025-5831,1432-1807},
     journal={Math. Ann.},
      volume={355},
      number={1},
       pages={57\ndash 95},
         url={https://doi.org/10.1007/s00208-011-0773-x},
      review={\MR{3004576}},
}

\bib{Tsai}{article}{
      author={Tsai, Yi-Lin},
       title={A concavity inequality for {H}essian quotient equations},
        date={2027},
        ISSN={0022-0396,1090-2732},
     journal={J. Differential Equations},
      volume={489},
       pages={Paper No. 114892},
         url={https://doi.org/10.1016/j.jde.2026.114892},
      review={\MR{5148561}},
}

\bib{Tu}{misc}{
      author={Tu, Qiang},
       title={Pogorelov type estimates for {$(n-1)$}-{H}essian equations and
  related rigidity theorems},
        date={preprint on arXiv},
  note={\href{https://arxiv.org/abs/2405.02939}{https://arxiv.org/abs/2405.02939}},
}

\bib{Urbas}{article}{
      author={Urbas, John I.~E.},
       title={On the existence of nonclassical solutions for two classes of
  fully nonlinear elliptic equations},
        date={1990},
        ISSN={0022-2518,1943-5258},
     journal={Indiana Univ. Math. J.},
      volume={39},
      number={2},
       pages={355\ndash 382},
         url={https://doi.org/10.1512/iumj.1990.39.39020},
      review={\MR{1089043}},
}

\bib{V}{article}{
      author={Viaclovsky, Jeff~A.},
       title={Conformal geometry, contact geometry, and the calculus of
  variations},
        date={2000},
        ISSN={0012-7094,1547-7398},
     journal={Duke Math. J.},
      volume={101},
      number={2},
       pages={283\ndash 316},
         url={https://doi.org/10.1215/S0012-7094-00-10127-5},
      review={\MR{1738176}},
}

\bib{Wang-Yuan}{article}{
      author={Wang, Dake},
      author={Yuan, Yu},
       title={Hessian estimates for special {L}agrangian equations with
  critical and supercritical phases in general dimensions},
        date={2014},
        ISSN={0002-9327,1080-6377},
     journal={Amer. J. Math.},
      volume={136},
      number={2},
       pages={481\ndash 499},
         url={https://doi.org/10.1353/ajm.2014.0009},
      review={\MR{3188067}},
}

\bib{WY}{article}{
      author={Warren, Micah},
      author={Yuan, Yu},
       title={Hessian estimates for the sigma-2 equation in dimension 3},
        date={2009},
        ISSN={0010-3640,1097-0312},
     journal={Comm. Pure Appl. Math.},
      volume={62},
      number={3},
       pages={305\ndash 321},
         url={https://doi.org/10.1002/cpa.20251},
      review={\MR{2487850}},
}

\bib{Zhang}{article}{
      author={Zhang, Ruijia},
       title={{$C^2$} estimates for {$k$}-{H}essian equations and a rigidity
  theorem},
        date={2025},
        ISSN={0001-8708,1090-2082},
     journal={Adv. Math.},
      volume={480},
       pages={Paper No. 110488, 35},
         url={https://doi.org/10.1016/j.aim.2025.110488},
      review={\MR{4948350}},
}

\bib{Zhou}{article}{
      author={Zhou, XingChen},
       title={Notes on generalized special {L}agrangian equation},
        date={2024},
        ISSN={0944-2669,1432-0835},
     journal={Calc. Var. Partial Differential Equations},
      volume={63},
      number={8},
       pages={Paper No. 197, 28},
         url={https://doi.org/10.1007/s00526-024-02801-w},
      review={\MR{4782212}},
}

\end{biblist}
\end{bibdiv}

\end{document}